\documentclass{article}
\usepackage[margin=1in,footskip=0.25in]{geometry}
\usepackage{authblk}
\usepackage{inputenc}
\usepackage{times,amsmath,amsthm,amssymb,thmtools,thm-restate}
\usepackage{bm}
\usepackage{thmtools,thm-restate,xargs,hyperref}
\usepackage{extarrows}
\usepackage[dvipsnames]{xcolor}
\usepackage{mathrsfs}
\usepackage{soul}
\usepackage[numbers]{natbib}

\newcommand{\UT}{\mathsf{T}}   
\declaretheorem{lemma}
\declaretheorem{proposition}
\declaretheorem{corollary}
\declaretheorem{definition}
\declaretheorem{remark}
\declaretheorem{fact}
\declaretheorem{observation}
\declaretheorem{theorem}

\newcommand{\E}{\mathbb{E}}
\newcommand{\hatE}{\hat{\E}}
\renewcommand{\P}{\mathbb{P}}
\newcommand{\Tr}{\mathsf{Tr}}

\newcommand{\N}{\mathbb N}
\newcommand{\W}{\mathbb{N}_0}
\newcommand{\R}{\mathbb R}
\newcommand{\C}{\mathbb C}
\newcommand*\diff{\mathop{}\!\mathrm{d}}
\newcommand{\gauss}[2]{\mathcal{N}\left( #1,#2 \right)}

\newcommand{\explain}[2]{\overset{\text{\tiny{#1}}}{#2}}

\newcommand{\Indicator}[1]{\mathbb{I} \left( #1 \right)}
\newcommand{\ip}[2]{\langle {#1}, {#2} \rangle}

\newcommand{\mult}[2]{\mathsf{Mult} \left( #1, #2 \right)}
\newcommand{\supp}[1]{\mathsf{Supp} \left( #1 \right)}
\newcommand{\pois}[1]{\mathsf{Pois} \left( #1 \right)}
\newcommand{\unif}[1]{\mathsf{Unif} \left( #1 \right)}

\newcommand{\vstat}[2]{\mathsf{VSTAT}_{#1}{\left(#2 \right)}}
\newcommandx{\sdn}[3][1=\text{$D_0$}, 3=\text{$\mathcal{D}(\pi)$}]{\mathsf{SDN}\text{$\left( #1, #3, #2 \right)$}}
\newcommand{\tpca}[3]{\mathsf{TPCA}(#1,#2,#3)}
\newcommand{\tensor}[2]{\bigotimes^{#2} #1}
\newcommandx{\Tsum}[3][3=]{\text{$\mathsf{SUM}_{#3}\left( #1, #2 \right)$}}
\newcommandx{\Tpar}[3][3=]{\text{$\mathsf{PARITY}_{#3}\left( #1, #2 \right)$}}
\newcommand{\Rank}{\mathsf{Rank}}
\newcommand{\oddness}{\mathsf{o}}

\newcommand{\basis}[1]{\mathsf{BASIS} \left( #1 \right)}
\newcommand{\BoolSpace}[1]{\mathcal{F}_{\mathsf{BOOLEAN}}({#1})}
\newcommand{\GaussSpace}[1]{\mathcal{L}_2 \left( \gauss{\bm 0}{\bm I_{#1}} \right)}
\newcommand{\smooth}[2]{\mathsf{SMOOTH}_{#2}[#1]}

\newcommand{\change}[1]{#1}

\begin{document}



\title{Statistical Query Lower Bounds for Tensor PCA}
\author[1]{Rishabh Dudeja\thanks{rd2714@columbia.edu}}
\author[2]{Daniel Hsu\thanks{djhsu@cs.columbia.edu}}
\affil[1]{Department of Statistics, Columbia University}
\affil[2]{Department of Computer Science, Columbia University}

\renewcommand\Authands{ and }

\maketitle

\begin{abstract}In the Tensor PCA problem introduced by \citet{richard2014statistical}, one is given a dataset consisting of $n$ samples $\bm T_{1:n}$ of i.i.d.~Gaussian tensors of order $k$ with the promise that $\E\bm T_1$ is a rank-1 tensor and $\|\E \bm T_1\| = 1$. The goal is to estimate $\E \bm T_1$. This problem exhibits a large conjectured hard phase when $k>2$: When $d \lesssim n \ll d^{\frac{k}{2}}$ it is information theoretically possible to estimate $\E \bm T_1$, but no polynomial time estimator is known. We provide a sharp analysis of the optimal sample complexity in the Statistical Query (SQ) model and show that SQ algorithms with polynomial query complexity not only fail to solve Tensor PCA in the conjectured hard phase, but also have a strictly sub-optimal sample complexity compared to some polynomial time estimators such as the  Richard-Montanari spectral estimator. Our analysis reveals that the optimal sample complexity in the SQ model depends on whether $\E \bm T_1$ is symmetric or not. For symmetric, even order tensors, we also isolate a sample size regime in which it is possible to test if $\E \bm T_1 = \bm 0$ or $\E \bm T_1 \neq \bm 0$ with polynomially many queries but not estimate $\E \bm T_1$. Our proofs rely on the Fourier analytic approach of \citet{feldman2018complexity}  to prove sharp SQ lower bounds.
\end{abstract}

\section{Introduction}

In the Tensor PCA testing problem, one is given a dataset consisting of $n$ i.i.d.~tensors $\bm T_{1:n} \in \tensor{\R^d}{k}, \; k \geq 2$, and the goal is to determine if the tensors were drawn from the null hypothesis $D_0$,
\begin{align*}
D_0:  \;  (T_{i})_{j_1,j_2,\dotsc, j_k} &\explain{i.i.d.}{\sim} \gauss{0}{1} \; \forall \; j_1,j_2,\dotsc, j_k \in [d], \; \forall \; i \in [n],
\end{align*} 
or from a distribution in the composite alternate hypothesis,
\begin{align*}
\mathcal{D} & = \{D_{\bm V}: \bm V \in \tensor{\R^d}{k}, \;  \|\bm V\| = \sqrt{d^k}, \; \Rank(\bm V) = 1\},
\end{align*}
where under the alternate hypothesis $D_{\bm V}$ the entries of the tensor are distributed as
\begin{align*}
D_{\bm V}: \; (T_{i})_{j_1,j_2, \dotsc, j_k} &\explain{i.i.d.}{\sim} \gauss{\frac{V_{j_1,j_2, \dotsc, j_k}}{d^{\frac{k}{2}}}}{1} \; \forall \; j_1,j_2 \dotsc, j_k \in [d], \; \forall \; i \in [n].
\end{align*}
In the Tensor PCA estimation problem, given a dataset $\bm T_{1:n}$ sampled i.i.d.~from an unknown distribution $D \in \mathcal{D}$, the goal is to estimate the \emph{signal tensor} $\E_D \bm T = \bm V / \sqrt{d^k}$.

When $k=2$, this model is called the Spiked Wigner Model and is often attributed to \citet{johnstone2001distribution}, who introduced a variant of this model as a non-null model for Principal Component Analysis. For general $k$, this model was introduced by \citet{richard2014statistical}\footnote{In the standard formulation of Tensor PCA introduced by \citet{richard2014statistical} one only observes the empirical average of the tensors $\overline{\bm T}$. However this is equivalent to the model studied here since given the $n$ samples $\bm T_{1:n}$, one can compute $\overline{\bm T}$, and given $\overline{\bm T}$, one can generate an i.i.d.~sample $\bm T_{1:n}$ without knowing $\E \bm T$ since $\overline{\bm T}$ is a sufficient statistic for $\E \bm T$.}, who envisioned this model as stylized setup to study statistical and computational phenomena occurring in higher order moment based estimators for fitting latent variable models \citep{anandkumar2014tensor}. In these applications, a suitable sample moment tensor can also be decomposed as a sum of a low rank population tensor, which contains information about parameters of interest, and a random tensor arising due to finite sample fluctuations.

The most striking feature of the Tensor PCA problem is that it exhibits a large gap between the information theoretically optimal performance and the performance achieved by computationally efficient estimators when $k>2$. Estimating the mean of a Gaussian tensor in general requires $n \asymp d^k$ samples. However, in the Tensor PCA problems since the mean is low rank, the information theoretic sample complexity requirement is $n \asymp d$. For example, when the signal tensor is symmetric, it is known that there exists a critical value $\beta^\star_k$ such that if $n < (\beta^\star_k - o_d(1)) \cdot d$, then it is information theoretically impossible to solve the Tensor PCA testing or estimation problems \citep{montanari2015limitation}. On the other hand, the results of \citet{richard2014statistical} imply the existence of another threshold $\beta^{\sharp}_k$ such that if $n > (\beta^{\sharp}_k + o_d(1)) \cdot d$, then tests and estimators based on the maximum likelihood value:
\begin{align}
    \max_{\|\bm x\| = 1} \overline{\bm T}(\bm x, \bm x, \dotsc ,\bm x) \label{MLE}, \hspace{0.4cm} \overline{\bm T} \explain{def}{=} \frac{1}{n} \sum_{i=1}^n \bm T_i,
\end{align}
have a non-trivial testing and estimation performance. When $k>2$, no computationally efficient algorithm to solve the optimization problem in Eq.~\ref{MLE} is known. \citeauthor{richard2014statistical} have proposed a computationally tractable relaxation of the optimization problem that flattens the mean tensor $\overline{\bm T}$ into a $d^{\lceil \frac{k}{2} \rceil} \times d^{\lfloor \frac{k}{2} \rfloor}$ matrix and uses the largest singular value and vector of the resulting matrix for testing and estimation respectively. This procedure requires $n \gtrsim d^{\frac{k}{2}}$ to solve the testing and estimation problems \citep{hopkins2015tensor,zheng2015interpolating}. To this date no known computationally efficient procedure improves on the performance of the Richard-Montanari spectral method. Consequently it is conjectured that any computationally efficient testing or estimation procedure requires $n \gtrsim d^{\frac{k}{2}}$ samples. \change{Proving this conjecture is expected to be more difficult than resolving P v.s. NP since proving computational lower bounds for average-case problems is harder than proving computational lower bounds for worst-case problems. However, recent works have exhibited some evidence that lends credibility to this conjecture. } 

In this paper, we study lower bounds on a class of algorithms for Tensor PCA: those that can be implemented in Statistical Query (SQ) computational model introduced by \citet{kearns1998efficient}. In the SQ Model, the estimation or testing procedure does not have direct, unrestricted access to the data $\bm T_{1:n}$ generated from the unknown distribution $D \in \mathcal{D} \cup \{D_0\}$, but instead can access the data only through an oracle. The algorithm can query the oracle with a statistic $q$ of interest, and the oracle returns an estimate of the expected value of the statistic $\E_D q(\bm T)$. The response of the oracle, denoted by $\hatE_D q(\bm T)$, has an error guarantee comparable to the error guarantee of the mean of the statistic on the dataset, but can be otherwise arbitrary (see Definition~\ref{VSTAT_def}); that is,
\begin{align}
    \left| \hatE_D q(\bm T) - \E_D q(\bm T) \right| & \approx \left| \frac{1}{n} \sum_{i=1}^n q(\bm T_i) - \E_D q(\bm T) \right|.
\label{VSTAT_error_guarantee_informal}
\end{align}
The SQ framework to study computational-statistical gaps aims to show that in the regime where no polynomial time estimators are known for the learning problem, no SQ algorithm can solve the learning problem after making only polynomially many queries. We emphasize that a successful SQ algorithm must work with an arbitrary oracle whose errors $|\hatE_D q(\bm T) - \E_D q(\bm T)|$ are possibly adversarial but do satisfy the requirement of Eq.~\ref{VSTAT_error_guarantee_informal}, and not just the most natural sample mean oracle which estimates $\E_D q(\bm T)$ by the empirical mean of $q(\cdot)$ on $n$ samples drawn from $D$. This strong robustness requirement which is expected in the SQ model is what makes it possible to prove strong lower bounds in this model.  At the same time, one hopes that though most practical procedures require direct access to the samples $\bm T_{1:n}$, they are essentially computing certain adaptively chosen statistics $q$ on the dataset. Moreover, one hopes that their analysis only relies on the fact that the sample mean of these statistics satisfies an error guarantee of the form given in Eq.~\ref{VSTAT_error_guarantee_informal}. It would then be possible to implement these procedures in the SQ model without any degradation in their performance and hence, lower bounds proved in the SQ model would apply to these procedures. 
 
 Indeed, these hopes have been realized in a number of problems which exhibit statistical-computational gaps. Two notable examples are the Planted Clique problem \citep{feldman2017statistical} and sparse PCA \citep{wang2015sharp}. Both of these problems exhibited a gap between the information theoretically optimal performance and the performance of the best known polynomial time estimator. The performance of the best known polynomial time estimator can be achieved in the SQ model and moreover is the best possible performance in the SQ model. Hence, for these problems the SQ Model provides a framework to prove lower bounds on the sample complexity of a class of computationally tractable estimators which include the best known estimators for this problem. Furthermore, the SQ lower bounds for these problems provide an intuitive \emph{mechanistic explanation} for the observed statistical-computational gap based on the informational structure of the problem: In order to declare that the unknown distribution $D \not \in \mathcal{D}$, an SQ algorithm must rule out all alternate hypothesis $D_1 \in \mathcal{D}$ by asking queries such that $|\hat\E_D q(\bm T) - \E_{D_1} q(\bm T)|$ is more than the finite sample fluctuations expected from $n$ samples. The SQ lower bound proceeds by exhibiting an \emph{exponentially} large (in $d$) collection of distributions $\mathcal{H} \subset \mathcal{D}$ such that even the single most powerful query is able to rule out only a small fraction of the distributions in $\mathcal{H}$ (more precisely, $o(d^{-t})$ fraction for any $t \in \mathbb N$). Consequently, any SQ algorithm must make a super-polynomial number of queries to solve these problems in the regimes where they are believed to be computationally hard. More recently, SQ lower bounds have also been used to reason about computational hardness of various robust estimation problems \citep{diakonikolas2017statistical,diakonikolas2019efficient}. We also refer the reader to the recent review article on the SQ model \citep{reyzin2020statistical}.

\subsection{Problem Definition and the Statistical Query Model}

Recall that in the alternative hypothesis for Tensor PCA, $\E \bm T = \bm V/\sqrt{d^{k}}$, where $\bm V$ is a rank-1 tensor with $\|\bm V\| = \sqrt{d^k}$. Hence we can write $\bm V$ as:
\begin{align*}
    \bm V & = \bm v_1 \otimes \bm v_2 \dots \otimes \bm v_k, \; \bm v_i \in \R^d, \; \|\bm v_i\| = \sqrt{d} \; \forall \;  i \; \in \; [k].
\end{align*}
In the standard symmetric version of Tensor PCA introduced by \citeauthor{richard2014statistical}, one further assumes $\bm v_1 = \bm v_2 = \dotsb = \bm v_k$. However, the completely asymmetric version where this constraint is not enforced is also common in the literature \citep{zheng2015interpolating,zhang2018tensor}. It turns out that the asymmetric Tensor PCA problems are harder to solve for SQ algorithms that make only polynomially many queries. In order to get more insight into this phenomena, in this paper, we study a generalization of Tensor PCA to incorporate partial symmetries where only some of the $\bm v_i$'s are identical. This interpolates between symmetric Tensor PCA on one end and completely asymmetric Tensor PCA on the other end. The generalization is parametrized by integers $k, K \in \N$ with $K \leq k$ and a labelling function $\pi: [k] \rightarrow [K]$ such that, under the alternate hypothesis,
\begin{align*}
    \E \bm T & = \frac{1}{\sqrt{d^k}} \cdot \bm v_{\pi(1)} \otimes \bm v_{\pi(2)} \otimes \dotsb \otimes \bm v_{\pi(k)}.
\end{align*}
For example, if $k=3,K=2$ and $\pi(1) = \pi(2) = 1, \pi(3) = 2$, then $\sqrt{d^3} \cdot \E \bm T = \bm v_1 \otimes \bm v_1 \otimes \bm v_2$, for two, possibly distinct $\bm v_1,\bm v_2 \in \R^d$ with $\|\bm v_1\| = \|\bm v_2\| = \sqrt{d}$.
Hence, we define the composite alternative hypothesis:
\begin{align*}
    \mathcal{D}(\pi)& = \{ \mathcal{D}_{\bm V}: \; \bm V  = \bm{v}_{\pi(1)} \otimes \bm{v}_{\pi(2)} \otimes \dotsb \otimes \bm{v}_{\pi(k)}, \; \bm{v}_{1:K} \in \R^d, \; \|\bm v_i\| = \sqrt{d} \; \forall \; i \; \in \; [K]   \}.
\end{align*}
In order to prove the strongest possible lower bounds, we will assume the algorithm knows the labelling function $\pi$. In the (generalized) Tensor PCA testing problem, the testing algorithm is  given data $\bm T_{1:n}$ sampled i.i.d.~from some unknown $D \in \{D_0\} \cup \mathcal{D}(\pi)$ and has to decide if $D  = D_0$ or $D \in \mathcal{D}(\pi)$. In the (generalized) Tensor PCA estimation problem, the data is generated from an unknown $D \in \mathcal{D}(\pi)$ and one seeks to estimate $\E_D \bm T$ with small error. 

We study tests and estimators for Tensor PCA in the Statistical Query framework, where the algorithm can access the data only through the VSTAT oracle introduced by \citet{feldman2017statistical}.
\begin{definition}[VSTAT oracle for Tensor PCA] \label{VSTAT_def} For any $D \in \{D_0\} \cup \mathcal{D}(\pi)$ and any $n \in \N$, \;  $\vstat{D}{n}$ is an oracle that, when given as input an arbitrary function $q: \tensor{\R^d}{k} \mapsto [0,1]$, returns an estimate $\hat \E_D q(\bm T)$ of the expectation $\E_D q(\bm T)$ with the error guarantee:
\begin{align}
| \hatE_D q(\bm T) - \E_D q(\bm T) | & \leq \max \left( \frac{1}{n} , \sqrt{\frac{\E_D q(\bm T) \cdot (1- \E_D q(\bm T))}{n}}\right). \label{VSTAT_error_guarantee}
\end{align}
The parameter $n$ is known as the \emph{effective sample size} of the oracle. 
\end{definition}
The rational for defining the oracle in this way is that the error guarantee of the oracle given in Eq.~\ref{VSTAT_error_guarantee} is comparable to the error guarantee obtained if one estimated $\E_D  q(\bm T)$ with its sample average on a dataset of size $n$ drawn from the distribution $D$. This is also why the parameter $n$ in $\vstat{D}{n}$ is interpreted as the effective sample size. Our goal is to understand how large must the effective sample size $n$ be for there to exist an SQ algorithm to solve the Tensor PCA testing problem by making only polynomial number of queries in the following sense:

\begin{definition}[Tensor PCA Testing in the SQ Model] \label{tester_def} An algorithm $\mathcal{A}$ \emph{solves the Tensor PCA testing problem in the SQ model with $n$ samples and $B$ queries} if, for any $D \in \{D_0\} \cup \mathcal{D}(\pi)$ (unknown to $\mathcal{A}$), given access to an arbitrary $\vstat{D}{n}$ oracle, $\mathcal{A}$ correctly decides if $D= D_0$ or $D \in \mathcal{D}(\pi)$  after making at most $B$ (possibly adaptive) queries to $\vstat{D}{n}$.
\end{definition}

\begin{definition}[Tensor PCA Estimation in the SQ Model] \label{estimator_def} An algorithm $\mathcal{A}$ \emph{solves the Tensor PCA estimation problem in the SQ model with $n$ samples and $B$ queries} if, for any $D \in \mathcal{D}(\pi)$ (unknown to $\mathcal{A}$), given access to an arbitrary $\vstat{D}{n}$ oracle, $\mathcal{A}$ returns an estimate $\hat{\bm V}$ such that
\begin{align*}
    \|\hat{\bm V} - \E_D \bm T\| & \leq \frac{1}{4}, 
\end{align*}
after making at most $B$ (possibly adaptive) queries to $\vstat{D}{n}$.
\end{definition}

The sample complexity of Tensor PCA in the Statistical Query depends on some structural properties of the labelling function $\pi$. For any $i \in [K]$, let $s_i$ be the number of times label $i$ is used in the labelling function $\pi$:
\begin{align}
    s_i & \explain{def}{=} |\pi^{-1}(i)| = |\{j \in [k]: \pi(j) = i\}|. \label{s_def}
\end{align}
The sample complexity of Tensor PCA depends crucially on the number of labels that are used an odd number of times. We denote this parameter by $\oddness$ and call it the \emph{oddness parameter} of the labelling function $\pi$, i.e., $\oddness = | \{i \in [K]: s_i \text{ is odd}\}|$. Note that $\oddness \in \{0,1,\dotsc, k\}$.

 \subsection{Our Contribution}
 Our main result is the following SQ lower bound for Tensor PCA.
  \begin{theorem} \label{main_result}
	Let $C_0 \geq 0$ and $\epsilon \in (0,1)$ be arbitrary constants. Consider the learning problems:
	\begin{enumerate}
	    \item Tensor PCA testing with $\oddness = 0$ and sample size $n \leq C_0 d^{\frac{k}{2} - \epsilon}$.
	    \item Tensor PCA estimation with $\oddness = 0$ and sample size $ n \leq C_0 d^{\frac{k+2}{2}-\epsilon}$
	    \item Tensor PCA testing or estimation with $\oddness \geq 1$ and sample size $n \leq C_0 d^{\frac{k+\oddness}{2} - \epsilon}$.
	\end{enumerate} 
	Then for any $L \in \N$, there exists a positive constant $c_1(k,L,\epsilon,C_0)>0$ depending only on $k,L,\epsilon,C_0$ and a finite constant $C_1(k,L,\epsilon)< \infty$ depending only on $k,L,\epsilon$ such that for all $d \geq C_1(k,L,\epsilon)$, any SQ algorithm that solves any of the above problems  must make at least $c_1(k,L,\epsilon, C_0) \cdot d^L$ queries.
\end{theorem}
We also design optimal SQ procedures, which show that the above lower bounds are tight. 
\begin{theorem} \label{upper_bound_thm} There exist SQ algorithms that make $O(d^k \cdot \log(n))$ queries and solve:
\begin{enumerate}
    \item The Tensor PCA testing problem with $\oddness = 0$ provided $n \gg \sqrt{d^k}$.
    \item The Tensor PCA estimation problem with $\oddness = 0$ provided $n \gg \sqrt{d^{k+2}}$.
    \item The Tensor PCA testing and estimation problems with $\oddness \geq 1$ provided $n \gg \sqrt{d^{k + \oddness}}$.
\end{enumerate}
\end{theorem}
Our results have the following implications:
\begin{enumerate}
    \item As mentioned previously, when $k\geq 3$, no known computationally efficient algorithm is able to leverage the low rank structure in Tensor PCA optimally.  Our results show that SQ algorithms also do not work in the conjectured hard phase $d \lesssim n \ll \sqrt{d^k}$.
\end{enumerate}
The SQ framework leverages the adversarial corruptions in the query responses to prove strong lower bounds. However, our result shows that these adversarial corruptions have the unintended effect of making the SQ lower bounds too pessimistic:
\begin{enumerate}
      \setcounter{enumi}{1}
    \item When $k=2$, the estimator based on the best rank-1 approximation to $\overline{\bm T}$ is able to effectively leverage the low dimensional structure and attain the information theoretic sample complexity in polynomial time. In contrast, SQ estimators with polynomial query complexity are ineffective at taking advantage of this low dimensional structure and require as many samples as in the case when the mean is unstructured ($n \gtrsim d^2$).
    \item When $k \geq 3$, SQ estimators with polynomially many queries are still less effective than known polynomial time algorithms in taking advantage of the low rank structure, however now, the effectiveness of SQ procedures depends crucially on the underlying symmetry of the mean tensor. They are least effective when the mean tensor is completely asymmetric, where they require as many samples as in the unstructured case ($n \gtrsim d^k$).
    \item  When $\oddness = 0$ (e.g., even-order symmetric Tensor PCA), the problem exhibits an estimation-testing gap in the SQ model: in the regime $\sqrt{d^k}  \lesssim n \ll \sqrt{d^{k+2}}$ testing is possible but estimation is not. No such testing-estimation gap is seen outside the SQ model.
 \end{enumerate}
 The ineffectiveness of SQ algorithms seems connected to the fact that leveraging the low rank structure requires solving stochastic non-convex optimization with adversarial, inexact gradient oracles. The results of \citet{feldman2017stochastic} show that stochastic \emph{convex} optimization problems based on $n$ samples (such as the least squares problem) can be solved with restricted access to a $\vstat{}{n}$ oracle with no degradation in performance. In contrast, our estimation lower bound for the  symmetric spiked Wigner problem ($k=2,\oddness = 0$, $\E \bm T = \bm v_0 \bm v_0^\UT$) exhibits a canonical, benign stochastic non-convex optimization problem: $\max_{\|\bm v\| = 1}\E \ip{\bm v}{\bm T \bm v}$ which cannot be solved with a $\vstat{}{n}$ oracle without a degradation in sample complexity compared to its empirical version: $\max_{\|\bm v\| = 1}\ip{\bm v}{\overline{\bm T} \bm v}$. This rules out  non-convex extensions to the results of \citet{feldman2017stochastic} without a worse dependence on problem parameters. This suboptimality of SQ algorithms  seems to arise because random initialization starts close to a saddle point of the population objective. In the spiked Wigner case, the $n$ sample empirical gradient at initialization $\bm v_\mathsf{init}$ is $2\ip{\bm v_0}{\bm v_\mathsf{init}} \bm v_0 + 2\bm g$ where $\bm g \sim \gauss{\bm 0}{\bm I_d/n}$. Although the ``signal'' part is small, $\ip{\bm v_0}{\bm v_{\mathsf{init}}} \approx d^{-\frac{1}{2}}$, the noise $\bm g$ is randomly oriented, $\ip{\bm g}{\bm v_0} \approx n^{-1/2}$, and hence when $n \gg d$, the signal part dominates the noise part in the projection of the gradient along $\bm v_0$ at initialization. This signal is amplified in the subsequent iterations of the power method. In contrast, when a $\vstat{}{n}$ oracle is used, the approximate gradient is of the form $2\ip{\bm v_0}{\bm v_\mathsf{init}} \bm v_0 + 2\tilde{\bm g}$, where $\|\tilde{\bm g}\|^2 \approx \|\bm g\|^2 \approx d/n$; but the noise vector $\tilde{\bm g}$ is no longer randomly oriented and can be adversarially chosen to cancel out the signal unless $n \gg d^2$. 

Given that SQ algorithms are suboptimal for Tensor PCA, one might wonder why one should care about SQ lower bounds for Tensor PCA.
  Indeed, the SQ prediction of the threshold for hardness of Tensor PCA is nowhere near the correct (conjectured) threshold, so SQ lower bounds are not evidence for the computational hardness of the problem.
We believe that SQ lower bounds for Tensor PCA are still useful for two reasons. (1) Currently, there are several frameworks to predict computational-statistical gaps including SQ lower bounds, Hopkins' Low Degree method {\citep{hopkins2018statistical,kunisky2019notes}}, reductions from other hard problems {\citep{berthet2013complexity}}, and statistical physics heuristics like failure of Belief Propagation {\citep{zdeborova2016statistical}}. We believe it is of interest to report successes and failures of these frameworks in canonical problems to understand their relative strengths and weaknesses. (2)  SQ lower bounds are also useful even beyond the study of computational-statistical gaps. SQ algorithms are robust since they can't use specialized properties of finite sample fluctuations implied by model assumptions. This is desirable in practice, and it is of interest to understand how the sample complexity changes under some natural deviations from model assumptions.
\subsection{Related Work}
A number of works have investigated the computational-statistical gap in Tensor PCA, which we discuss below.

\paragraph{Studies of Computational-Statistical Gap in Tensor PCA:} \citet{arous2019landscape} have shown that in the information theoretic regime $n \asymp d$, the expected number of local maximizers of the maximum likelihood optimization problem (Eq.~\ref{MLE}) is exponentially large. These local maximizers can potentially trap gradient based methods. \citet{arous2018algorithmic} have shown that when $n \gg d^{k-1}$, Langevin Dynamics succeeds in solving Tensor PCA and that this is essentially tight for Langevin Dynamics to work. Sum of Squares relaxations for Tensor PCA have been analyzed \citep{hopkins2015tensor,hopkins2016fast,hopkins2017power,bhattiprolu2016sum} and it is known that computationally efficient sum-of-squares based procedures match the performance of the Richard-Montanari spectral estimator but fail to improve on it. {A popular way to provide evidence for hardness of a particular problem is via reductions, i.e., to show that solving it would lead to algorithms to solve another problem which is widely believed to be hard. Reductions were introduced to reason about the hardness of statistical inference problems in the seminal work of \citet{berthet2013complexity}. In the specific case of Tensor PCA, it has been shown that hardness of Hypergraph Planted Clique implies the hardness of Tensor PCA \citep{zhang2018tensor,brennan2020reducibility}.}

Our results reveal a gap between the optimal sample complexity of Tensor PCA for polynomial query SQ algorithms and computationally efficient algorithms. A well known example of this situation is efficient learning of parity functions (in a noise-free setting), which is possible using the Gaussian elimination algorithm, but not using SQ algorithms \citep{blum1994weakly}. Apart from this, we are aware of at least two other instances where SQ algorithms are suboptimal in comparison to some computationally efficient and noise tolerant procedure, which we discuss next.

\paragraph{Planted Satisfiability:} \citet{feldman2018complexity} study a broad class of planted random planted constraint satisfaction problems in which an unknown assignment $\bm \sigma_\star \in \{0,1\}^d$ is fixed and we observe $n$ clauses $C_1, C_2, \dotsc, C_n$ drawn independently and uniformly at random from a family of clauses $\mathcal{C}$, conditioned on $\bm \sigma_\star$ being a satisfying assignment for the clause. The goal is to identify the planted assignment $\bm \sigma_\star$.  When $\mathcal{C}$ is the set of all XOR expressions involving $k$ variables, the problem is called the planted $k$-XOR-SAT problem and exhibits a large computational-statistical gap: observing $n \gtrsim d \log(d)$ clauses is information theoretically sufficient to identify the planted assignment $\bm \sigma_\star$ with high probability. However, all noise-tolerant and  computationally efficient methods require $ n \gtrsim d^{\frac{k}{2}}$ clauses. On the other hand, \citet{feldman2018complexity} show that SQ algorithms fail unless $n \gtrsim d^{k}$. Moreover, this work introduces the Fourier analytic approach to prove SQ lower bounds that we rely on. Previous approaches for proving SQ lower bounds \citep{feldman2017statistical} are only able to show a weaker sample complexity lower bounds of $n \gtrsim d^{\frac{k}{2}}$ for $k$-XOR-SAT as well as $k$-Tensor PCA. {This work also proposes a hierarchical generalization  of the $\vstat{}{n}$ oracle for satisfiability problems known as the $\mathsf{MVSTAT}(n,\ell)$ oracle where $\ell$ is a parameter controlling the strength of the oracle. The oracle is designed such that, for a well-chosen value of $\ell$, it is possible to implement a particular $d^{\frac{k}{2}}$ sample complexity estimator for $k$-XOR-SAT without any degradation in performance}. 

\paragraph{Matrix Mean Estimation:} The recent work \citet{li2019mean} studies mean estimation on general normed spaces. Their results on matrix mean estimation with respect to the operator norm ($\| \cdot \|_\mathsf{op}$) are particularly relevant for our problem. In this problem, one is given $n$ samples $\bm X_1, \bm X_2, \dotsc, \bm X_n$ of $d \times d$ matrices sampled i.i.d.~from some distribution supported on unit operator norm ball, and the goal is to estimate $\E \bm X$ with an error of $\epsilon = 0.1$ in operator norm. By the matrix Bernstein inequality~\citep{tropp2015introduction}, when $n \gtrsim \log(d)$, the empirical average of the matrices $\overline{\bm X}$ satisfies $\|\overline{\bm X} - \E \bm X\|_\mathsf{op} \leq 0.1$. On the other hand, \citeauthor{li2019mean} show that in the SQ model one requires $n \gtrsim d$ samples to solve this problem. Though their results are stated for unstructured matrix mean estimation, the hard distribution they construct is supported on highly sparse Boolean matrices with a low rank expectation. Thus, after a suitable change of scale, their hard instance also shows $n \gtrsim d^2$ sample complexity lower bound in the SQ model for a non-Gaussian variant of low rank matrix mean estimation. Their proof also relies on the Fourier analytic approach of \citet{feldman2018complexity}.

Another commonly used framework to provide evidence for a computational-statistical gap is the low-degree polynomial framework, which we discuss next.

\paragraph{Connection with low-degree polynomial lower bounds:} \change{The low-degree polynomial framework seeks to provide evidence for a computational-statistical gap by showing that any procedure that seeks to solve the inference problem by computing a low-degree polynomial of the data fails to solve the inference in the conjectured hard phase \citep{hopkins2018statistical,kunisky2019notes}. This framework consistently reproduces widely believed average-case hardness conjectures for many problems, including Tensor PCA \citep{hopkins2017power,hopkins2018statistical,kunisky2019notes}. The concurrent work of \citet{brennan2020statistical} identifies general conditions under which lower bounds for low-degree polynomial procedures are equivalent to a \emph{sufficient criterion} for SQ lower bounds, namely the statistical dimension criterion of \citet{feldman2017statistical}. Hence, for many inference problems, including the completely symmetric Tensor PCA testing problem, they obtain SQ lower bounds from previously known low-degree lower bounds. However, the SQ lower bounds obtained in this manner are often not tight. For example, \citeauthor{brennan2020statistical} show that when $n \ll d^{\frac{3}{2}}$, SQ algorithms with polynomial query complexity fail to solve the  completely symmetric Tensor PCA testing problem with $k=3$. In contrast, our results show that the same conclusion holds when $n \ll d^{2}$.  The reason for this appears to be that the statistical dimension criterion of \citet{feldman2017statistical} studied by \citeauthor{brennan2020statistical} is only a sufficient (and not necessary) criterion for SQ lower bounds. In contrast, our lower bounds rely on a different criterion introduced by \citet{feldman2018complexity} which is known to imply stronger SQ lower bounds in other problems like planted CSPs \citep{feldman2018complexity}. }

\subsection{Notation}
In this section, we introduce the notation used in this paper.

\paragraph{Common Sets:} $\R, \N, \W$ denote the set of real numbers, natural numbers and non-negative integers respectively. We denote the finite set $\{1,2 \dots, k\}$ by $[k]$.  The set $\tensor{\R^d}{k}$ denotes the set of all $d \times d \times \dotsb \times d$ tensors with real entries. The set $\tensor{\W^d}{k}$ denotes the set of all $ d \times d \times \dotsb \times d$ tensors with non-negative integer entries. 

\paragraph{Vectors, Matrices and Tensors:} We denote vectors, matrices and tensors with bold face letters. We denote the $d$ dimension vectors $(1,1, \dotsc,1)$ and $(0,0,\dotsc, 0)$ as $\bm 1_d$ and $\bm 0_d$. We will omit the subscript $d$ when it is clear from the context.  For vectors, $\|\cdot \|$ and $\|\cdot \|_1$ denote the Euclidean ($\ell_2$) norm and the $\ell_1$ norm respectively. For vectors $\bm a , \bm b \in \R^d$, $\ip{\bm a}{\bm b}$ denotes the standard Euclidean inner product $\sum_{i=1}^d a_i b_i$. These norms and inner product are analogously defined for matrices and tensors by stacking their entries into a single vector. For a tensor $\bm T \in \tensor{\W^d}{k}$, we define the map $\Tsum{\bm T}{l}: \tensor{\W^d}{k} \times [k] \mapsto \R^d$ which computes the partial sum of the entries of $\bm T$ along all modes except mode $l$:
	\begin{align}
	\label{tsum_notation}
	\Tsum{\bm T}{l }_i & = \sum_{j_1,j_2 \dots j_{l-1}, j_{l+1} \dots j_k \in [d]} T_{j_1,j_2 \dots j_{l-1},i, j_{l+1} \dots j_k}.
	\end{align}
	Likewise we define $\Tpar{\cdot}{\cdot }: \tensor{\W^d}{k} \times [k] \mapsto \{0,1\}^d$ as follows:
	\begin{align}
	\label{tparity_notation}
	\Tpar{\bm T}{l}_i & = \bigoplus_{j_1,j_2 \dots j_{l-1}, j_{l+1} \dots j_k \in [d]} T_{j_1,j_2 \dots j_{l-1},i, j_{l+1} \dots j_k}.
	\end{align}
In the above display, $\oplus$ denotes addition modulo $2$.

Given a labelling function $\pi : [k] \mapsto [K]$, which assigns each mode $l \in [k]$ to a label $\pi(l) \in K$, we can combine the partial sums which leave out modes with the same label by adding them. This defines a map $\Tsum{\bm T}{i}[\pi]: \tensor{\W^d}{k}\times [K] \mapsto \R^d$:
\begin{align*}
    \Tsum{\bm T}{i}[\pi] & = \sum_{l: \pi(l) = i} \Tsum{\bm T}{l}.
\end{align*}
Analogously, we combine the partial parities which leave out modes with the same label via the map $\Tpar{\bm T}{i}[\pi]: \tensor{\W^d}{k} \times [K] \mapsto \{0,1\}^d$ defined as follows:
\begin{align*}
        \Tpar{\bm T}{i}[\pi] & = \bigoplus_{l: \pi(l) = i} \Tpar{\bm T}{l}.
\end{align*}
For a tensor $\bm T$, the support of $\bm T$ is $\supp{\bm T} = \{(i_1,i_2, \dotsc, i_k) \in \N^k: T_{i_1,i_2, \dotsc, i_k} \neq 0\}$.

\paragraph{Special Distributions:} $\gauss{\bm \mu}{\bm \Sigma}$ denotes the multivariate Gaussian distribution with mean $\bm \mu$ and covariance matrix $\bm \Sigma$. $\pois{\lambda}$ denotes the Poisson distribution with rate $\lambda \geq 0$. $\mult{n}{k}$ denotes the Multinomial distribution governing the random vector that counts the occurrences of $1, 2, \dotsc, k$ when a fair $k$-faced die is independently thrown $n$ times. For a finite set $S$, $\unif{S}$ denotes the uniform distribution on $S$, and $\unif{\{\pm 1\}}$ is also called the Rademacher distribution. 

\paragraph{Asymptotic Order Notation:} We use the following notation to describe the behaviour of positive sequence $f(d)$ and $g(d)$ for large $d$. We say $f(d) \lesssim g(d)$ if there exists  a constant $c$ independent of $d$ such that $f(d) \leq c g(d)$. If $f(d) \lesssim g(d)$ and $g(d) \lesssim f(d)$, then we say $f(d) \asymp g(d)$. Finally, we say a sequence $f(d) \ll g(d)$ if there exists an $\epsilon > 0$ such that $f(d) d^{\epsilon} \lesssim g(d)$.

\paragraph{Miscellaneous:} For real numbers $a,b$, $a \vee b$ denotes the maximum of $a,b$. For a vector $\bm v \in \R^d$ and a vector $\bm c \in \W^d$, we define the entry-wise factorial and the entry-wise powering notations:
	\begin{align}
	\label{entry_wise_notation}
	\bm c ! \explain{def}{=} \prod_{i=1}^d c_{i}!, \; \bm a^{\bm c} \explain{def}{=} \prod_{i=1}^d  a_i^{c_i}.
	\end{align}
	We also use the natural extension of the above entry-wise operators to the case when $\bm a, \bm c$ are tensors. 
	
\section{Proof Overview for Theorem \ref{main_result}}
For Tensor PCA with the labelling function $\pi$, we consider the following prior on $\bm v_1, \bm v_2, \cdots, \bm v_K$: $\bm v_i \explain{i.i.d.}{\sim} \unif{\{\pm 1 \}^d}$. Define the random tensor $\bm V = \bm v_{\pi(1)} \otimes \bm v_{\pi(2)} \otimes \cdots \otimes \bm v_{\pi(k)}$. Note that for any realization of $\bm V$, $D_{\bm V} \in \mathcal{D}(\pi)$. Standard results on the SQ framework (developed in \citet{feldman2018complexity}, reviewed in Appendix~\ref{framework_appendix}) show that the minimal possible query complexity of an SQ algorithm that solves the Tensor PCA testing problem with $n$ samples is determined by the asymptotics of the following quantity:
\begin{align}
    \max_{q: \E_{D_0} q^2(\bm T) = 1} \Delta(q; D_0), \text{ where } \; \Delta(q;D_0) \explain{def}{=} \P_{\bm V} \left[ \left| \E_{D_{\bm V}} q(\bm T) - \E_{D_0} q(\bm T) \right| \gtrsim \frac{1}{\sqrt{n}}\right]. \label{main_object_of_study}
\end{align}
The intuition for this is as follows. Consider a fixed realization of the r.v.~$\bm V$. Suppose for a query $q$, we have $|\E_{D_{\bm V}} q(\bm T) - \E_{D_0} q(\bm T) |  \gtrsim n^{-\frac{1}{2}}$. Then given a response from an arbitrary $\vstat{D}{n}$ oracle, one can distinguish if $D = D_0$ or if $D=D_{\bm V}$ since the gap $|\E_{D_{\bm V}} q(\bm T) - \E_{D_0} q(\bm T) |$ is more than the size of the error introduced by the $\vstat{D}{n}$ oracle. In this situation, we say that the query $q$ has \emph{resolved} the alternative hypothesis $D_{\bm V}$. In order to solve the composite hypothesis testing problem of distinguishing $D = D_0$ or $D \in \mathcal{D}(\pi)$, one must ask queries to resolve at least every hypothesis  $D_{\bm V}$ for every possible realization of $\bm V$. Note that for a query $q$, the quantity $\Delta(q; D_0)$ measures the information content of a single query: the fraction of realizations of $\bm V$ that are resolved by the query $q$. In order to minimize the number of queries asked, one should ask powerful queries: single queries which resolve a large fraction of hypotheses in one go. Hence, in order to show that no SQ algorithm can solve Tensor PCA  with polynomially many queries it is sufficient to show that any single query resolves at most $o(d^{-t})$ fraction of possible realizations of $\bm V$ $\forall \; t \in \N$.

\change{For the estimation problem, the results of \citet[Theorem 4.2]{feldman2017general} show that the number of queries that need to be made to solve the estimation problem with $n$ samples is governed by the asymptotics of a quantity analogous to object in Eq.~\ref{main_object_of_study} with a small change:} the measure $D_0$ can be replaced by an arbitrary measure on $D_\star$ on $\tensor{\R^d}{k}$. For most cases, choosing $D_\star = D_0$ suffices. In other cases we will choose  
\begin{align}
D_\star = \overline{D} (\pi) \explain{def}{=} D_{\overline{\bm V}(\pi)}, \;\text{where } \overline{ \bm V}(\pi) \explain{def}{=} \E \bm V.
\label{vbar_def}
\end{align}

To analyze Eq.~\ref{main_object_of_study}, we use a key idea of \citet{feldman2018complexity}: $\E_{D_{\bm v_{\pi(1)} \otimes \dotsb \otimes \bm v_{\pi(k)}}} q(\bm T) - \E_{D_0} q(\bm T)$ is a function of Boolean random vectors $\bm v_1, \bm v_2, \dotsc, \bm v_K$, and Hypercontractivity Theorems can be used to control its moments and concentration properties. Implementing this approach requires an understanding of the Fourier spectrum of $\E_{D_{\bm v_{\pi(1)} \otimes \dotsb \otimes \bm v_{\pi(k)}}} q(\bm T) - \E_{D_0} q(\bm T)$, which is possible due to the additive Gaussian structure of Tensor PCA. This structure is also leveraged in the Low Degree method for predicting computational-statistical gaps \citep{hopkins2018statistical,kunisky2019notes}. This results in the following proposition, which is proved in Appendix \ref{fourier_analytic_prop_proof}.
\begin{proposition} \label{fourier_analytic_prop}
	Let $\epsilon > 0, \; u \in \N, u \geq 2$ be arbitrary. For the reference distribution $D_0$, and for any query $q$ with $\E_{D_0} q^2(\bm T) = 1$, we have,
	\begin{align*}
	 \P_{\bm V} \left[ \left| \E_{D} q(\bm T) - \E_{D_0} q(\bm T) \right| > \epsilon \right] & \leq \left( \frac{e}{\epsilon^2} \cdot \max_{l_1,l_2,\dotsc,l_K \in \W} \left(  (u-1)^{l_1 + l_2 + \dotsb +  l_K} \cdot p_\pi(l_{1:K}) \right) \right)^{\frac{u}{2}},
	\end{align*}
	where the coefficients $p_\pi(l_1,l_2, \dotsc, l_K)$ are defined as follows:
	\begin{align*}
        p_\pi(l_1,l_2,\dotsc,l_K) & \explain{def}{ =} \frac{1}{e} \sum_{\substack{\bm c \in \tensor{\W^d}{k}:  \|\bm c\|_1 \geq 1 \\ \Tpar{\bm c}{i}[\pi] =(\bm{1}_{l_i}, \bm 0)\; \forall \;  i \;  \in \;  [K]}} \frac{1}{d^{k \|\bm c\|_1}\cdot \bm c!}.
    \end{align*}
    For the reference measure $\overline{D}(\pi)$, for any query $q$ such that $\E_{\overline{D}(\pi)} q^2(\bm T) = 1$, we have,
    \begin{align*}
        \P_{\bm V} \left[ \left| \E_{D} q(\bm T) - \E_{\overline{D}(\pi)} q(\bm T) \right| > \epsilon \right] & \leq \left( \frac{e}{\epsilon^2} \cdot \max_{l_1,l_2,\dotsc,l_K \in \W} \left(  (u-1)^{l_1 + l_2 \dots + l_K} \cdot \overline{p}_\pi(l_{1:K}) \right) \right)^{\frac{u}{2}}, 
    \end{align*}
    where the coefficients $\overline{p}_\pi(l_1,l_2, \dotsc, l_K)$ are defined as follows:
	\begin{align*}
        \overline{p}_\pi(l_1,l_2,\dotsc, l_K) & \explain{def}{ =} \frac{1}{e} \sum_{\substack{\bm c \in \tensor{\W^d}{k}:  \|\bm c\|_1 \geq 1 \\ \Tpar{\bm c}{i}[\pi] =(\bm{1}_{l_i}, \bm 0)\; \forall \;  i \;  \in \;  [K]}} \frac{1}{d^{k \|\bm c\|_1}\cdot \bm c!} \cdot \Indicator{\supp{\bm c} \cap \supp{\overline{\bm V}(\pi)} = \phi}.
    \end{align*}
\end{proposition}

\subsection{Asymptotics of Combinatorial Coefficients $p_\pi$ and $\overline{p}_\pi$}
Proving sharp SQ lower bounds for Tensor PCA requires a tight asymptotic analysis of the combinatorial coefficients $p_\pi$ and $\overline{p}_\pi$. In order to illustrate why this requires a delicate analysis, we consider the following natural bound:
\begin{align}
    p_\pi(l_1,l_2, \dots ,l_K) =  \frac{1}{e} \sum_{\substack{\bm c \in \tensor{\W^d}{k}:  \|\bm c\|_1 \geq 1 \\ \Tpar{\bm c}{i}[\pi] =(\bm{1}_{l_i}, \bm 0)\; \forall \;  i \;  \in \;  [K]}} \frac{1}{d^{k \|\bm c\|_1}\cdot \bm c!} 
    \leq \frac{1}{ e d^k} \sum_{\substack{\bm c \in \tensor{\W^d}{k}}} \frac{1}{\bm c!}  
    =  \frac{1}{e d^k} e^{d^k} \label{silly_upperbound},
\end{align}
where the last step follows from $e = 1 + 1/1! + 1/2! + \dotsb$. This estimate diverges exponentially with $d^k$ and is too weak for our analysis. In order to obtain tight estimates, we rely on a natural a probabilistic interpretation of the coefficients $p_\pi$ as a certain Poisson probability. Let $\bm C$ be a random tensor in $\tensor{\W^d}{k}$ whose entries are sampled independently as $C_{i_1,i_2, \dotsc, i_k}  \explain{i.i.d.}{\sim} \pois{d^{-k}}$.
Recalling the formula for the probability mass function of the Poisson distribution (Fact \ref{poisson_distribution_definition}), we have
\begin{align*}
    p_\pi(l_1,l_2,\dotsc, l_K) & \explain{def}{ =} \sum_{\substack{\bm c:  \|\bm c\|_1 \geq 1 \\ \Tpar{\bm c}{i}[\pi] =(\bm{1}_{l_i}, \bm 0)\; \forall \;  i \;  \in \;  [K]}} \prod_{i_1,i_2,\dotsc, i_k \in [d]} \frac{e^{-\frac{1}{d^k}}}{(d^{k})^{c_{i_1,\dotsc, i_k}} \cdot c_{i_1,\dotsc, i_k}!} \\
    & = \P \left( \|\bm C\|_1 \geq 1, \Tpar{\bm C}{i}[\pi] = (\bm{1}_{l_i}, \bm 0) \; \forall \; i \; \in [K]  \right).
\end{align*}
Analogously,
\begin{multline*}
     \overline{p}_\pi(l_1,l_2,\dotsc, l_K) \\ = \P \left( \|\bm C\|_1 \geq 1, \Tpar{\bm C}{i}[\pi] = (\bm{1}_{l_i}, \bm 0) \; \forall \; i \; \in [K], \; \supp{\bm C} \cap \supp{\overline{\bm V}(\pi)} = \phi  \right).
\end{multline*}
This interpretation already shows that $p_\pi \leq 1$ and $\overline{p}_\pi \leq 1$, which is much better than the trivial estimate in Eq.~\ref{silly_upperbound}. We obtain tighter bounds by taking into account the parity constraints using the following conditional independence structure in Poisson random tensors.
\begin{restatable}{proposition}{poistensorprop}  \label{poisson_tensor_proposition}
Let $\bm C$ be a Poisson random tensor sampled as above. Then, we have
\begin{align*}
    \|\bm C\|_1 & \sim \pois{1} .
\end{align*}
Moreover, conditioned on $\|\bm C\|_1$, we have
\begin{align*}
    \Tsum{\bm C}{i} \;  \big| \;  \| \bm C\|_1 & \explain{i.i.d.}{\sim} \mult{\| \bm C\|_1}{d}, \; \forall \; i \; \in \;  [k].
\end{align*}
\end{restatable}
The proof of this property can be found in Appendix \ref{Poisson_tensors_appendix}. The following result provides an exact formula for the probabilities $\P(\Tpar{\bm C}{i} = (\bm{1}_{l_i}, \bm 0) \; \forall \; i \; \in [K])$ which in principle allows us to compute the complete asymptotic expansion for these probabilities for large $d$.

\begin{restatable}{lemma}{rep_lemma_rad} 
Let $\bm X_1, \bm X_2, \dotsc, \bm X_K$ be $K$ independent vectors in $\{\pm 1\}^d$ whose entries $(X_i)_j$  are i.i.d.~Rademacher distributed and independent of $\bm C$. Let $\overline{X}_i$ denote the mean of the entries of the random vector $\bm X_i$.
Then,
\begin{align}
    \P\left( \Tpar{\bm C}{i}[\pi] = (\bm{1}_{l_i}, \bm 0) \; \forall \; i \; \in [K] \right) &= \frac{1}{e} \cdot \E \left[ e^{\overline{X}_1^{s_1} \cdot \overline{X}_2^{s_2} \dotsm \overline{X}_K^{s_K}} \prod_{i=1}^K \prod_{j=1}^{l_i} X_{ij}  \right] \label{representation_formula_lemma_eq},
\end{align}
where $s_{1:K}$ were defined previously in Eq.~\ref{s_def}.
\label{representation_formula_rademacher}
\end{restatable}
The proof of this lemma uses the following simple observation. Let $Y$ be an arbitrary random variable taking values in $\W$. Let $X$ be a Rademacher random variable independent of $Y$. Then,
\begin{align*}
    \P(\text{Y has even parity}) = \E [X^Y], \; \P(\text{Y has odd parity})  = \E [X^{Y+1} | Y] .
\end{align*}
Appealing to this observation, we have,
\begin{align*}
    \P \left( \Tpar{\bm C}{i}[\pi] = (\bm{1}_{l_i}, \bm 0) \; \forall \; i \; \in [
    K] \right)& = \E \left[ \bm X_1^{\Tsum{\bm C}{1}[\pi]} \dotsm \bm X_K^{\Tsum{\bm C}{K}[\pi]} \cdot \prod_{i=1}^K \prod_{j=1}^{l_i} X_{ij} \right].
\end{align*}
The expectation with respect to $\bm C$ can now be evaluated by conditioning on $\|\bm C\|_1$ and using Proposition \ref{poisson_tensor_proposition} along with the formula for the moment generating function for the Poisson and Multinomial distributions. A complete proof is provided in Appendix~\ref{rademacher_formula_proof}. The advantage of the formula given in Lemma \ref{representation_formula_rademacher} is that it reduces the analysis of a combinatorial expression to controlling moments of subgaussian random variables. By expanding the exponential in Eq.~\ref{representation_formula_lemma_eq} to a suitable order and using the fact that $\sqrt{d} \cdot  \overline{X}_i$ are 1-subgaussian results in the following sharp upper bounds on the combinatorial coefficients $p_\pi(l_1,l_2 \dots l_K)$ and $\overline{p}_\pi(l_1,l_2 \dots l_K)$. 
\begin{restatable}{proposition}{asympprop} \label{asymptotic_proposition}
The coefficients $p_\pi(l_1,l_2 \dots l_K)$ satisfy the following estimates:
\begin{align*}
    p_\pi(l_1,l_2,\dotsc, l_K) & \leq \begin{cases} C_k \cdot d^{-\frac{k}{2}} &: l_1 = l_2 = \dotsb = l_K = 0, \; \oddness = 0 \\
    C_k \cdot d^{-k} &: l_1 = l_2 = \dotsb = l_K = 0, \; \oddness \geq 1 \\
    C_k \cdot d^{-\frac{k + \sum_{i=1}^K l_i}{2}}  &: l_i \leq |\pi^{-1}(i)|, \; l_i + |\pi^{-1}(i)| \text{ is even } \forall \; i \; \in [K] \\
    C_k \cdot d^{-\frac{l_1 \vee l_2 \vee \dotsb \vee l_K}{2}}&: l_1 \vee l_2 \vee \dotsb \vee l_K \geq 2k\\
    C_k \cdot  d^{-k} &: \text{otherwise}.
    \end{cases}
\end{align*}
For any $l_1,l_2 \dotsc, l_K \in \W$, we have, $\overline{p}_\pi(l_1,l_2 , \dotsc, l_K)  \leq p_\pi(l_1,l_2 , \dotsc, l_K)$ and, in the special case $l_1 = l_2 = \dotsb = l_K = 0$, we have, $\overline{p}_\pi(0,0,\dotsc, 0) \leq \frac{C_k}{d^k}$. In the above equations, $C_k$ denotes a universal constant that depends only on $k$.
\end{restatable}

The proof of this Proposition can be found in Appendix~\ref{asymptotics}. Theorem \ref{main_result} now follows immediately by substituting the above upper bounds on $p_{\pi}$ and $\overline{p}_\pi$ in Proposition~\ref{fourier_analytic_prop} with $\epsilon \asymp n^{-\frac{1}{2}}$. More details can be found in Appendix~\ref{main_result_proof}.

\section{Optimal Statistical Query Procedures}
Theorem \ref{upper_bound_thm} claims the existence of SQ algorithms which achieve the lower bounds proved in Theorem \ref{main_result}. A complete description of these procedures can be found in Appendix \ref{optimal_SQ_appendix}. Interestingly, these procedures do not try and solve either the maximum likelihood problem or its spectral relaxation. Instead, these procedures leverage the ``Partial Trace'' technique used by \citet{hopkins2016fast,anandkumar2017homotopy,biroli2019iron} to compress the empirical mean tensor in symmetric Tensor PCA into a vector or a matrix, and incur only a modest decrease in the signal-to-noise ratio. To implement these estimators in the SQ model, we leverage on the following result.

\begin{fact}[SQ Scalar Mean Estimation; \citealp{feldman2017dealing}] \label{scalar_mean_estimation_fact} There exists a SQ algorithm that given parameters $\xi,B > 0$, a query function $q: \tensor{\R^d}{k} \mapsto \R$  and access to an arbitrary $\vstat{D}{n}$ oracle for any distribution $D \in \mathcal{D}(\pi) \cup \{D_0\}$ such that $\E_D q^2(\bm T) \leq B$,  returns an estimate of $\E_D q(\bm T)$ denoted by $\hatE_D q(\bm T)$ such that:
\begin{align*}
    \left|\hatE_D q(\bm T) - \E_D q(\bm T) \right| & \leq 8 \cdot \log(n) \cdot \sqrt{\frac{\operatorname{Var}_D q(\bm T)}{n}} + \xi.
\end{align*}
after making at most $3 \log(4 n B/ \xi^2)$ queries. 
\end{fact}
We describe the optimal SQ algorithms in the following illustrative special cases.

\paragraph{Even order, Symmetric Tensor PCA Testing:} In this case, $k = 2l$ for some $l \in \N$ and $K=1$. The procedure uses Fact \ref{scalar_mean_estimation_fact} to estimate the mean of the query $$q(\bm T) = \bm T(\bm I_d, \bm I_d ,\dotsc, \bm I_d) \explain{def}{=} \sum_{j_1,j_2 \dots ,j_l \in [d]} T_{j_1,j_1,j_2,j_2 \dots ,j_l,j_l}.$$ It is easy to check that when $D$ is in the alternative hypothesis, $q(\bm T) \sim \gauss{1}{\sqrt{d^k}}$ and when $D = D_0$, then $q(\bm T) \sim \gauss{0}{\sqrt{d^k}}$. Consequently when $n \gg \sqrt{d^k}$, the error guarantee of the estimate from Fact \ref{scalar_mean_estimation_fact} is small enough to ensure that null and alternate hypothesis are distinguishable.  

\paragraph{Completely Asymmetric Tensor PCA:} In this case $K = k$. The optimal procedure simply estimates the mean of each entry of the tensor using Fact \ref{scalar_mean_estimation_fact}. The resulting tensor of mean estimates is used as an estimate of the signal tensor for the estimation problem and its norm is used to solve the testing problem. Since the errors of a $\vstat{}{n}$ oracle can be adversarial, the error in each entry can add up. Consequently, this approach works only when $n \gg d^k$.

\paragraph{Partially Symmetric Tensor PCA:} The general case can be handled by a combination of the above techniques. At a high level, the signal tensor can be written as a outer product of a compressible part and an incompressible part. In the first step, a partial trace is used to compress the tensor by eliminating the compressible part. The mean of the incompressible part is estimated entry-wise using Fact \ref{scalar_mean_estimation_fact}. The reduction in the size of the tensor during the compression phase ensures that the accumulation of adversarial SQ errors is minimized. The estimate of the incompressible part can be used to construct a test or to recover the compressible part of the signal tensor. 

\section{Effect of Noise Variance}
In the formulation of Tensor PCA that we study in this paper, we assume that the noise variance is $1$. It is natural to study formulations where the noise variance $\sigma^2$ is a parameter that is allowed to depend on $n,d$. More formally, in the testing problem $\tpca{n}{d}{\sigma^2}$, one is given a dataset consisting of $n$ i.i.d.~tensors $\bm T_{1:n} \in \tensor{\R^d}{k}, \; k \geq 2$, and the goal is to determine if the tensors were drawn from the null hypothesis $D_0$,
\begin{align*}
D_0(\sigma^2):  \;  (T_{i})_{j_1,j_2,\dotsc, j_k} &\explain{i.i.d.}{\sim} \gauss{0}{\sigma^2} \; \forall \; j_1,j_2,\dotsc, j_k \in [d], \; \forall \; i \in [n],
\end{align*} 
or from a distribution in the composite alternate hypothesis,
\begin{align*}
\mathcal{D}(\sigma^2) & = \{D_{\bm V}(\sigma^2): \bm V \in \tensor{\R^d}{k}, \;  \|\bm V\| = \sqrt{d^k}, \; \Rank(\bm V) = 1\},
\end{align*}
where under the alternate hypothesis $D_{\bm V}$ the entries of the tensor are distributed as
\begin{align*}
D_{\bm V}(\sigma^2): \; (T_{i})_{j_1,j_2, \dotsc, j_k} &\explain{i.i.d.}{\sim} \gauss{\frac{V_{j_1,j_2, \dotsc, j_k}}{d^{\frac{k}{2}}}}{\sigma^2} \; \forall \; j_1,j_2 \dotsc, j_k \in [d], \; \forall \; i \in [n].
\end{align*}
In the estimation problem $\tpca{n}{d}{\sigma^2}$, given a dataset $\bm T_{1:n}$ sampled i.i.d.~from an unknown distribution $D \in \mathcal{D}(\sigma^2)$, the goal is to estimate the \emph{signal tensor} $\E_D \bm T = \bm V / \sqrt{d^k}$. When allowed unrestricted access to the dataset $\tpca{n_1}{d_1}{\sigma_1^2}$ and $\tpca{n_2}{d}{\sigma_2^2}$ testing and estimation problems are statistically and computationally equivalent if $\sigma_1^2/n_1 = \sigma_2^2/n_2$ because of the following observation. 

\begin{observation} \label{obs: sufficiency} Given $n_1$ samples drawn i.i.d.\ from $D_{\bm V}(\sigma_1^2)$ one can generate $n_2$ samples drawn i.i.d.\ from $D_{\bm V}(\sigma_2^2)$ without the knowledge of $\bm V$ if $\sigma_1^2/n_1 = \sigma_2^2/n_2$. Moreover, this sampling mechanism is computationally efficient. 
\end{observation}
\begin{proof} Note that one can make the transformation $(n_1,\sigma_1^2) \rightarrow (n_2,\sigma_2^2)$ in two steps: $(n_1,\sigma_1^2) \rightarrow (1,\sigma_1^2/n_1) = (1,\sigma_2^2/n_2) \rightarrow (n_2,\sigma_2^2)$. Hence, it is sufficient to show that given $n$ samples from $D_{\bm V}(\sigma^2)$ one can generate 1 sample from $D_{\bm V}(\sigma^2/n)$ and vice-versa. Note that if $\bm T_{1:n} \explain{i.i.d.}{\sim} D_{\bm V}(\sigma^2)$ then,
\begin{align*}
    \overline{T} & \explain{def}{=} \frac{1}{n} \sum_{i=1}^n \bm T_i \sim D_{\bm V}\left( \frac{\sigma^2}{n}\right).
\end{align*}
Hence, to generate a sample from $D_{\bm V}(\sigma^2/n)$ one can compute the empirical average of the $n$ samples from $D_{\bm V}(\sigma^2)$. For the other direction, we observe that since the empirical average is a sufficient statistic for the mean parameter $\bm V$ of a Gaussian distribution (with known variance), the conditional law of $\bm T_{1:n}$ given $\overline{\bm T}$ does not depend on the mean parameter $\bm V$. Hence, given one sample from $D_{\bm V}(\sigma^2/n)$ (i.e., $\overline{\bm T}$) one can generate $n$ samples  $\bm T_{1:n}$ by sampling from this conditional law without the knowledge of $\bm V$. Furthermore, since $\bm T_{1:n},\overline{\bm T}$ are jointly Gaussian, this conditional law is Gaussian and can be efficiently sampled.
\end{proof}

As a consequence of Observation \ref{obs: sufficiency}, the problems $\tpca{n_1}{d_1}{\sigma_1^2}$ and $\tpca{n_2}{d}{\sigma_2^2}$ are statistically and computationally equivalent: given an algorithm  $\mathcal{A}_1$ for $\tpca{n_1}{d}{\sigma_1^2}$ one can construct an algorithm $\mathcal{A}_2$ for $\tpca{n_2}{d}{\sigma_2^2}$ as follows: First use the reduction outlined in Observation \ref{obs: sufficiency} to construct a $\tpca{n_1}{d}{\sigma_1^2}$ instance from the $\tpca{n_2}{d}{\sigma_2^2}$ instance and then run $\mathcal{A}_1$ on the $\tpca{n_1}{d}{\sigma_1^2}$ instance. The procedure $\mathcal{A}_2$ has the same statistical performance on a $\tpca{n_2}{d}{\sigma_2^2}$ instance as procedure $\mathcal{A}_1$ has on $\tpca{n_1}{d}{\sigma_1^2}$ instance. Moreover if $\mathcal{A}_1$ is computationally efficient so is $\mathcal{A}_2$.

Given the above observation, a natural question is whether the above equivalence is preserved when one can only access the data via an SQ oracle. \change{Formally, let us consider} a pair of distributions $(D_1, D_2)$ such that:
\begin{align*}
    D_1 \in \{D_0(\sigma_1^2)\} \cup \mathcal{D}(\sigma_1^2), \; D_2 \in \{D_0(\sigma_2^2)\} \cup \mathcal{D}(\sigma_2^2), \; \E_{D_1} \bm T = \E_{D_2} \bm T.
\end{align*}
One can ask the question: Given access to a $\vstat{D_1}{n_1}$ can one simulate a $\vstat{D_2}{n_2}$ oracle and vice-versa? This would allow us to use SQ algorithms designed for $\tpca{n_2}{d}{\sigma_2^2}$ to solve $\tpca{n_1}{d}{\sigma_1^2}$ and vice-versa and lower bounds proved for one problem would have implications for the other problem. 

\begin{observation} \label{obs: simulation_feasible_side} Suppose that $\sigma_2^2 = \sigma_1^2 \cdot S$ for some $S \in \N$. Then given access to a $\vstat{D_1}{n_1}$ oracle can one simulate a $\vstat{D_2}{n_2}$ oracle with:
\begin{align*}
    n_2 & = \frac{S n_1}{256 \log^2(n_1)} = \frac{1}{256 \log^2(n_1)}\cdot \frac{n_1 \sigma_2^2}{\sigma_1^2}.
\end{align*}
Furthermore, this simulation makes at most $9\log(n_1S)$ queries to the $\vstat{D_1}{n_1}$ oracle to respond to one query for the $\vstat{D_2}{n_2}$ oracle. 
\end{observation}
The proof relies on a conditioning argument to reduce the variance of a statistic while preserving its expectation. This is often called Rao-Blackwellization \citep{rao1992information,blackwell1947conditional} in statistics. 

\begin{proof}
Let $q :  \tensor{\R^d}{k} \rightarrow [0,1]$ be the query for the $\vstat{D_2}{n_2}$ oracle. We will construct a query $Q$ for the $\vstat{D_1}{n_1}$ oracle with the following properties:
\begin{enumerate}
    \item Unbiasedness: $\E_{D_1}Q(\bm T) \explain{}{=} \E_{D_2} q(\bm T)$.
    \item  Variance bound: $\mathsf{Var}_{D_1} Q(\bm T) \explain{}{\leq} {\mathsf{Var}_{D_2}q(\bm T)}/{S}$.
\end{enumerate}
Let us assume for the moment that such a query $Q$ can be constructed and complete the proof of the claim. Indeed, by Fact~\ref{scalar_mean_estimation_fact} one can compute an estimate $\hat{Q}$ by making at most $9\log(n_1 S)$ queries to a $\vstat{D_1}{n_1}$ oracle which has the following error guarantee (we set $\xi = (2n_2)^{-1}$):
\begin{align*}
    \left| \hat{Q} - \E_{D_1} Q(\bm T) \right| & \leq  8 \log(n_1) \cdot  \sqrt{\frac{\mathsf{Var}_{D_1}(Q(\bm T))}{n_1}} + \xi \\ &\leq 2 \max\left( 8 \log(n_1) \cdot  \sqrt{\frac{\mathsf{Var}_{D_1}(Q(\bm T))}{n_1}}, \xi \right) \\& \leq \max\left( \sqrt{\frac{\mathsf{Var}_{D_2}(q(\bm T))}{n_2}} , \frac{1}{n_2}\right).
\end{align*}
Since $\mathsf{Var}_{D_2} q(\bm T)  = \E_{D_2} q(\bm T)^2 - (\E_{D_2} q(\bm T))^2 \leq (\E_{D_2} q(\bm T))\cdot (\E_{D_2} q(\bm T) -1)$, $\hat{Q}$ functions as a valid $\vstat{D_2}{n_2}$ response, completing the simulation of a $\vstat{D_2}{n_2}$ oracle. 

In order to finish the proof, we need to describe the construction of the query $Q$ with Properties 1 and 2 from above. In order to describe the construction, it will be helpful to realize a sample from $D_1$ and a sample from $D_2$ in the same probability space. We use the following natural construction: Let $\bm T_{1:S} \explain{i.i.d.}{\sim} D_2$ and define $\overline{\bm T} = \frac{1}{S} \sum_{i=1}^S \bm T_i$. Note that $\overline{\bm T} \sim D_1$. Note that, in order to construct the query $Q$, we need to find a function $Q$ of $\overline{\bm T}$ with the properties:
\begin{enumerate}
    \item Unbiasedness: $\E[Q(\overline{\bm T})] \explain{}{=} \E[q(\bm T_1)]$,
    \item  Variance bound: $\mathsf{Var}[Q(\overline{\bm T})] \explain{}{\leq}{\mathsf{Var} [q(\bm T_1)}]/{S}$.
\end{enumerate}
Note that the function:
\begin{align*}
    \overline{q}(\bm T_{1}, \bm T_2 \cdots ,\bm T_S) & \explain{def}{=} \frac{1}{S} \sum_{i=1}^S q(\bm T_i).
\end{align*}
satisfies Properties 1 and 2. However it is not a function of $\overline{\bm T}$. In order to fix this, define:
\begin{align*}
    Q(\overline{\bm T}) \explain{def}{=} \E \left[ \frac{1}{S} \sum_{i=1}^S q(\bm T_i) \bigg| \overline{\bm T}\right].
\end{align*}
This function $Q$ satisfies Properties 1 and 2:
\begin{enumerate}
    \item By the tower property:
    \begin{align*}
        \E Q( \overline{ \bm T}) = \E\left[ \frac{1}{S} \sum_{i=1}^S q(\bm T_i) \right] = \E q(\bm T_1).
    \end{align*}
    \item Due to the fact that conditioning reduces variance (or Jensen's inequality):
    \begin{align*}
        \mathsf{Var}[Q(\overline{\bm T})] & \leq  \mathsf{Var}\left[ \frac{1}{S} \sum_{i=1}^S q(\bm T_i) \right] \\
        & = \frac{\mathsf{Var} [q(\bm T_1)]}{S}.
    \end{align*}
\end{enumerate}
Hence, we have constructed the desired function $Q$ and this concludes the proof of the observation. 
\end{proof}

As a consequence of Observation \ref{obs: simulation_feasible_side}, the SQ lower bounds proved for $\tpca{n}{d}{1}$ in Theorem \ref{main_result} can be lifted to lower bounds for $\tpca{n}{d}{\sigma^2}$ provided $\sigma^2 \geq \Omega_d(1)$. 

\begin{corollary} \label{corr_large_noise} Let $\sigma^2 \geq 1$. The following learning problems: 
	\begin{enumerate}
	    \item $\tpca{n}{d}{\sigma^2}$ testing with $\oddness = 0$ and sample size $n \ll \sigma^2 d^{\frac{k}{2}}$.
	    \item $\tpca{n}{d}{\sigma^2}$ estimation with $\oddness = 0$ and sample size $ n \ll \sigma^2 d^{\frac{k+2}{2}}$
	    \item $\tpca{n}{d}{\sigma^2}$ testing or estimation with $\oddness \geq 1$ and sample size $n \ll  \sigma^2 d^{\frac{k+\oddness}{2}}$.
	\end{enumerate} 
	cannot be solved in the SQ model with a query complexity which is polynomial in $d$.
Furthermore, there exist SQ algorithms that make polynomial in $d$ queries and solve:
\begin{enumerate}
    \item The Tensor PCA testing problem with $\oddness = 0$ provided $n \gg \sigma^2\sqrt{d^k}$.
    \item The Tensor PCA estimation problem with $\oddness = 0$ provided $n \gg \sigma^2\sqrt{d^{k+2}}$.
    \item The Tensor PCA testing and estimation problems with $\oddness \geq 1$ provided $n \gg \sigma^2\sqrt{d^{k + \oddness}}$.
\end{enumerate} 
\end{corollary}
\begin{proof} The lower bounds follow by a proof by contradiction: if there is a polynomial query SQ algorithm which beats the claimed lower bounds for $\tpca{n}{d}{\sigma^2}$, using the simulation argument in Observation \ref{obs: simulation_feasible_side}, one can use it to solve $\tpca{n}{d}{1}$  with a sample complexity $n$ which beats the lower bounds in Theorem \ref{main_result}. For the upper bound, it is straightforward to extend the analysis of the optimal SQ procedures described in Section \ref{optimal_SQ_appendix} to arbitrary noise level $\sigma^2$.
\end{proof}

\begin{remark} One can show that Corollary \ref{corr_large_noise} holds for any $\sigma^2 \geq \Omega_d(1)$ (i.e., the constant $1$ is not special). This is because the proof of Theorem \ref{main_result} works for any constant $\sigma^2$ (not necessarily equal to $1$).  
\end{remark}

Observation \ref{obs: simulation_feasible_side} shows that an oracle with a larger sample size and a larger noise level is no stronger than an oracle with a smaller sample size and a proportionally smaller noise level. However the converse of this does not hold. 

\begin{observation} There is a constant $c>0$ small enough such that it is possible to solve symmetric $\tpca{n}{d}{\sigma^2}$ testing problem with $\sigma^2 = c/d$ in the SQ model with sample complexity $n = O_d(1)$ after making $O_d(1)$ queries.  
\end{observation}
\begin{proof}
Consider the query $q$:
\begin{align*}
    q(\bm T) & = \max_{\|\bm x \| = 1} \sum_{i_1,i_2 \dots i_k} T_{i_1,i_2 \dots i_k} x_{i_1} x_{i_2} \dots x_{i_k}. 
\end{align*}
Lemma 2.1 in \citet{richard2014statistical} shows that if $\sigma^2 = c/d$ for a small enough constant $c$, then  we have:
\begin{align*}
    |\E_D q(\bm T) - \E_{D_0(\sigma^2)} q(\bm T) | & \geq 0.5 \; \forall \; D \;  \in \; \mathcal{D}(\sigma_2).
\end{align*}
Furthermore Lemma 2.1 and 2.2 of \citeauthor{richard2014statistical} show that:
\begin{align*}
    \mathsf{Var}_{D} q(\bm T) & \leq \frac{C}{d}, \;  \forall \; D \;  \in \; \mathcal{D}(\sigma_2) \cup \{D_0(\sigma^2)\}.
\end{align*}
for some constant $C>0$. Consequently, by Fact \ref{scalar_mean_estimation_fact} we can obtain an estimate $\hatE_D q(\bm T)$ such that $|\hatE_D q(\bm T) - \E_D q(\bm T)| \leq 0.1$ after making $O_d(1)$ queries to a $\vstat{D}{n}$ oracle with $n = O_d(1)$ and hence solve the testing problem. 
\end{proof}

The above observation shows that the assumption that $\sigma_2^2 \geq 1$ cannot be completely removed Corollary \ref{corr_large_noise}. Note that $n \geq \Omega_d(1)$ is information theoretic requirement on the sample size to solve $\tpca{n}{d}{\sigma^2}$ testing problem with $\sigma^2 = \Theta_d(1/d)$ and no known polynomial-time estimator achieves this sample complexity. Hence when $\sigma^2 = \Theta_d(1/d)$, SQ algorithms are so powerful that they completely fail to capture the conjectured computational-statistical gap and it is, in fact, possible to implement computationally intractable estimators in the SQ model. 

The discussion in this section suggests that there is a possibility that for some well chosen setting of $1/d \ll \sigma^2 \ll 1$, the SQ lower bound may match the conjectured statistical-computational gap for all variants of Tensor PCA. We have not pursued this direction in our work because of the following reasons:
\begin{enumerate}
    \item We believe that the setting $\sigma^2 = \Theta_d(1)$ is the most natural setting to study Tensor PCA in the SQ model. This is because Tensor PCA was intended to be a stylized setup to understand tensor based methods to fit latent variable models \citep{richard2014statistical}. For example, consider estimating parameter $\bm \mu$  given data $$\bm x_{1:n} \explain{i.i.d}{\sim} 0.5 \mathcal{N}(-\bm \mu,\sigma^2\bm I_d) + 0.5 \mathcal{N}(\bm \mu,\sigma^2\bm I_d).$$ In these applications one observes several i.i.d. samples such that the signal-to-noise ratio per sample, measured by $\|\bm \mu\|^2/\sigma^2$ is $\Theta_d(1)$ and a sample size of  $n \asymp d$ (the dimension of the parameter) is information theoretically necessary and sufficient. In Tensor PCA since the signal tensor has norm $\|\E \bm T \| = 1$, this suggests the natural setting of $\sigma^2 = \Theta_d(1)$ which is why we chose to study the $\tpca{n}{d}{1}$ formulation. In contrast $\tpca{n}{d}{1/d}$ seems very different from the situation arising in applications since the signal-to-noise ratio per sample diverges with $d$.
    \item Setting $\sigma^2$ in a post-hoc manner to reproduce a widely believed computational-statistical gap does not seem like a satisfactory fix to the SQ model since this doesn't provide a framework to \emph{predict} computational-statistical gaps in new problems. 
\end{enumerate}
\section{Conclusions and Future Work}
In this paper, we studied the Tensor PCA problem in the Statistical Query model. We found that SQ algorithms not only fail to solve Tensor PCA in the conjectured hard phase, but they also have a strictly sub-optimal complexity compared to some polynomial time estimators such as the Richard-Montanari spectral estimator. Our analysis revealed that the optimal sample complexity in the SQ model depends on whether $\E \bm T$ is symmetric or not. Furthermore, we isolated a sample size regime in which only testing, but not estimation, is possible for the symmetric and even-order Tensor PCA problem in the SQ model.

\change{Our results show that the statistical query framework leads to an incorrect and pessimistic prediction of the statistical-computation gap in Tensor PCA. An interesting problem for future work is to address this limitation of the statistical query framework. \citet{feldman2018complexity} found that the SQ framework leads to a similar pessimistic prediction of the statistical-computational gap in planted constraint satisfaction problems (CSPs). To address this, \citeauthor{feldman2018complexity} proposed a hierarchical generalization of the $\vstat{}{n}$ oracle for satisfiability problems known as the $\mathsf{MVSTAT}(n,\ell)$ oracle where $\ell$ is a parameter controlling the strength of the oracle. The oracle is designed such that, for a well-chosen value of $\ell$, it is possible to implement spectral estimators that achieve the best-known performance among known polynomial-time estimators. An interesting problem for future work is to identify a similar generalization of the  $\vstat{}{n}$ oracle for Tensor PCA. It would be exciting if this generalization is not tailored to a particular inference problem and works for a broad class of inference problems. Another desirable outcome is that this generalization is parameter-free, i.e., does not require setting parameters governing the strength of the oracle in a post-hoc manner to reproduce a widely believed computational-statistical gap. A modification to the SQ framework with these two properties would enable us to \emph{predict} computational-statistical gaps in new problems rather than reproduce them. The recent work of \citet{brennan2020statistical} may provide hints towards identifying the correct modification to the SQ framework. This work shows that under some weak conditions, lower bounds for low-degree polynomial procedures (which are known to predict computational-statistical gaps correctly) are equivalent to a \emph{sufficient criterion} for SQ lower bounds, namely the statistical dimension criterion of \citet{feldman2017statistical}. An interesting problem suggested by their work is to identify a relaxation of the current adversarial noise model used in the SQ framework under which the statistical dimension criterion of \citet{feldman2017statistical} is a \emph{necessary and sufficient} condition for SQ lower bounds. Such a modification to the SQ model would make SQ lower bounds equivalent to low-degree lower bounds, which are known to reliably predict the computational-statistical gap in a wide range of inference problems.} 

\section*{Acknowledgements}
The authors would like to thank Tselil Schramm and Alex Wein for helpful discussions. DH was partially supported by NSF awards CCF-1740833, DMR-1534910, and IIS-1563785, a Google Faculty Research Award, and a Sloan Fellowship.


\appendix

\section{Framework for Statistical Query Lower Bounds} \label{framework_appendix}
In this section, we review a few standard definitions and results used to prove SQ Lower bounds.  First we review the notion of statistical dimension defined by \citet{feldman2018complexity}.
\begin{definition}[Statistical Dimension; \citealp{feldman2018complexity}] \label{sdn_definition} Let $\epsilon > 0$ be an arbitrary tolerance parameter. Let $D_\star$ be a fixed reference distribution and $\mathcal{H}$ be a finite collection of distributions such that $D_\star \not\in \mathcal{H}$ The statistical dimension at tolerance $\epsilon$ for $D_\star$ v.s.~$\mathcal{H}$ denoted by $\sdn[D_\star]{\epsilon}[\mathcal{H}]$ is the largest natural number $m$  with the property: for any subset $\mathcal{S} \subset \mathcal{H}$ with size at least $|\mathcal{S}| \geq |\mathcal{H}|/m$, we have,
\begin{align}
\label{sdn_implication}
 \frac{1}{|\mathcal{S}|} \sum_{D\in \mathcal{S}} \left| \E_{D} q(\bm T) - \E_{D_\star} q(\bm T) \right| \leq \epsilon \; \forall \; q: \tensor{\R^d}{k} \mapsto \R, \;  \E_{D_\star} q^2(\bm T)  = 1.
\end{align} 
\end{definition}
In our lower bounds, we will consider the following natural candidate for $\mathcal{H}$:
\begin{align*}
  \mathcal{H}(\pi) & \explain{def}{=} \{D_{\bm v_{\pi(1)} \otimes \bm v_{\pi(2)} \cdots \bm v_{\pi(k)}}: \bm v_{1:K} \in \{\pm 1\}^d \}.
\end{align*}
Our  lower bounds rely on the following result  which shows that the statistical dimension  is a lower bound on the query complexity of any SQ algorithm. 
\begin{proposition}[\citealp{feldman2018complexity,feldman2017general}] \label{feldmans_sq_lb} Consider the Tensor PCA problem with labelling function $\pi$. Let $\mathcal{H}$ be any arbitrary finite subset of $\mathcal{D}(\pi)$. Any  algorithm that solves the testing problem must make at least $\sdn[D_0]{(3n)^{-\frac{1}{2}}}[\mathcal{H}(\pi)]$ queries. Furthermore, any SQ algorithm that solves the estimation problem must make at least $0.5 \cdot \sdn[D_\star]{(3n)^{-\frac{1}{2}}}[\mathcal{H}(\pi)]$ queries, where $D_\star$ is an arbitrary reference measure (not necessarily $D_0$).
\end{proposition}
The claim of the above proposition for the testing problem is due to \citet[Theorem 7.1]{feldman2018complexity}.  \change{The claim regarding the estimation problem is due to \citet[Theorem 4.2]{feldman2017general}. A proof is provided in Appendix \ref{lb_on_sdn_appendix} for completeness.}

Note in order to obtain lower bounds for the estimation problem  we are free to choose $D_\star$. The natural choice of $D_\star = D_0$ will be sufficient to obtain tight lower bounds for most cases. For the remaining cases we will choose $D_\star = \overline{D} (\pi)$ which is a natural `average' of all distributions in $\mathcal{H}(\pi)$.
\begin{align*}
    \overline{D}(\pi) & \explain{def}{=} D_{\overline{\bm V}(\pi)}, \; \overline{\bm V}(\pi) = \E_{\bm v_1, \bm v_2 \cdots, \bm v_K \explain{i.i.d.}{\sim} \unif{\{\pm 1\}^d}} \left[ \bm v_{\pi(1)} \otimes \bm v_{\pi(2)} \cdots \bm v_{\pi(k)}\right].
\end{align*}

The following lemma presents a lower bound on the statistical dimension. Similar results are implicit in various previous works on SQ lower bounds \citep{feldman2018complexity}. 
\begin{lemma} \label{lb_on_sdn} For any $\epsilon > 0 $ and any $D_\star \in \{D_0, \overline{D}(\pi)\}$, 
	\begin{align*}
	\sdn[D_\star]{\epsilon}[\mathcal{H}(\pi)] & \geq \frac{\epsilon}{4} \cdot \left( \sup_{q: \E_{D_\star} q^2(\bm T) = 1} \P_{D \sim \unif{\mathcal H(\pi)}} \left[ \left| \E_{D} q(\bm T) - \E_{D_\star} q(\bm T) \right| > \frac{\epsilon}{2}\right] \right)^{-1}.
	\end{align*}
\end{lemma}
The proof of the above lemma can also be found in Appendix \ref{lb_on_sdn_appendix}. 

\section{Proof of Proposition \ref{fourier_analytic_prop}}
\label{fourier_analytic_prop_proof}
As a consequence of Proposition \ref{feldmans_sq_lb} and Lemma \ref{lb_on_sdn} from the previous section, in order to prove a lower bound on the number of queries required to solve Tensor PCA with $n$ samples, we need to prove an upper bound on the quantity:
\begin{align*}
    \sup_{q: \E_{D_\star} q^2(\bm T) = 1} \P_{D \sim \unif{\mathcal H(\pi)}} \left[ \left| \E_{D} q(\bm T) - \E_{D_\star} q(\bm T) \right| > \frac{1}{2\sqrt{3n}}\right]. 
\end{align*}
This section is devoted to the proof of Proposition \ref{fourier_analytic_prop} which provides an analysis of the above quantity.

\begin{proof}
	Consider any query $q$ such that $\E_{D_0} q^2(\bm T)  = 1$.
	Observe that we have the following distributional equivalence:
	\begin{align*}
	D \sim \unif{\mathcal H(\pi)} & \explain{d}{=} D_{\bm v_{\pi(1)} \dots \otimes \bm v_{\pi(k)}}, \; \bm v_i \explain{i.i.d.}{\sim} \unif{\{\pm 1\}^d}.
	\end{align*}
	Hence, 
	\begin{align*}
	&\P_{D \sim \unif{\mathcal H}} \left[ \left| \E_{D} q(\bm T) - \E_{D_0} q(\bm T) \right| > \epsilon \right]  = \\ & \hspace{5cm} \P_{\bm v_{1:K} \sim \unif{\{\pm 1\}^d}} \left[ \left| \E_{D_{\bm v_{\pi(1)} \dots \otimes \bm v_{\pi(k)}}} q(\bm T) - \E_{D_0} q(\bm T) \right| > \epsilon \right].
	\end{align*}
	In the remainder of the proof, for the ease of notation, we will shorthand $\E_{\bm v_{1:K} \sim \unif{\{\pm 1\}^d}}$ and $\P_{\bm v_{1:K} \sim \unif{\{\pm 1\}^d}}$ simply as $\E$ and $\P$.
	
	\paragraph{Step 1, Markov Inequality:} By Markov's Inequality, we have, for any $u \in \N$, 
	\begin{align}
	\label{markov_inequality}
	\P\left[ \left| \E_{D_{\bm v_{\pi(1)} \dots \otimes \bm v_{\pi(k)}}} q(\bm T) - \E_{D_0} q(\bm T) \right| > \epsilon \right] & \leq \frac{\E \left| \E_{D_{\bm v_{\pi(1)} \dots \otimes \bm v_{\pi(k)}}} q(\bm T) - \E_{D_0} q(\bm T) \right|^{u}}{\epsilon^u}.
	\end{align}
	
	\paragraph{Step 2, Fourier Expansion of $\E_{D_{\bm v_{\pi(1)} \dots \otimes \bm v_{\pi(k)}}} q(\bm T) - \E_{D_0} q(\bm T)$:} Since the second moment of $q$ is bounded, we can expand the query $q$ in terms of the multivariate Hermite polynomials:
	\begin{align*}
	q(\bm T) & = \sum_{\bm c \in \tensor{\W^d}{k}} \hat{q}(\bm c) H_{\bm c}(\bm T).
	\end{align*} 
	We refer the reader to Appendix \ref{fourier_gauss_appendix} for definitions and basic properties of this orthonormal basis. By Parseval's identity, we have,
	\begin{align*}
	\E_{D_0} q^2(\bm T) = 1 & \Leftrightarrow  \sum_{\bm c \in \tensor{\W^d}{k}} \hat q^2(\bm c) = 1.
	\end{align*}
	The rational of expanding the query in this particular basis is the following elegant property (see Fact \ref{hermite_special_property}, Appendix \ref{fourier_gauss_appendix})\footnote{This property is also useful in the analysis of Tensor PCA via the low degree method \citep{hopkins2017power,hopkins2018statistical,kunisky2019notes}, a heuristic technique motivated by Sum-of-Squares hierarchy to predict computational-statistical gaps.}:
	\begin{align*}
	\E_{D_{\bm v_{\pi(1)} \dots \otimes \bm v_{\pi(k)}}} H_{\bm c}(\bm T) & =  \prod_{j_1, j_2 \dots j_k \in [d]}  \left( \frac{ (v_{\pi(1)})_{j_1} (v_{\pi(2)})_{j_2} \dots (v_{\pi(k)})_{j_k}}{d^{\frac{k}{2}}} \right)^{c_{j_1,j_2 \dots j_k}} \cdot \frac{1}{\sqrt{c_{j_1,j_2 \dots j_k}!}}.
	\end{align*}
	After rearranging the terms in the product and recalling the $\Tsum{\cdot}{\cdot}, \Tpar{\cdot}{\cdot}$ and entry-wise factorial and powering notations from Eqs.~\ref{tsum_notation}, \ref{tparity_notation} and \ref{entry_wise_notation} one obtains the following concise representation of the above formula:
	\begin{align*}
	\E_{D_{\bm v_{\pi(1)} \dots \otimes \bm v_{\pi(k)}}} H_{\bm c}(\bm T) & = \frac{1}{\sqrt{d^{k \|\bm c\|_1} \cdot  \bm c!}} \cdot  \bm v_{\pi(1)}^{\Tsum{\bm c}{1}} \cdot \bm v_{\pi(2)}^{\Tsum{\bm c}{2}} \dots \cdot \bm v_{\pi(k)}^{\Tsum{\bm c}{k}} \\
	& =  \frac{1}{\sqrt{d^{k \|\bm c\|_1} \cdot  \bm c!}} \cdot  \bm v_{1}^{\Tsum{\bm c}{1}[\pi]} \cdot \bm v_{2}^{\Tsum{\bm c}{2}[\pi]} \dots \cdot \bm v_{K}^{\Tsum{\bm c}{K}[\pi]} \\
	& = \frac{1}{\sqrt{d^{k \|\bm c\|_1} \cdot  \bm c!}} \cdot  \bm v_1^{\Tpar{\bm c}{1}[\pi]} \cdot \bm v_2^{\Tpar{\bm c}{2}[\pi]} \dots \cdot \bm v_K^{\Tpar{\bm c}{K}[\pi]}.
	\end{align*}
	In the last step, we used the fact that $\bm v_i \in \{\pm 1\}^d, \; \forall \; i \in [K]$. Substituting this in the Fourier expansion of $\bm q$, one obtains the following representation of the bias $\E_{D_{\bm v_{\pi(1)} \cdots \otimes \bm v_{\pi(k)}}} H_{\bm c}(\bm T) - \E_{D_0} q(\bm T)$ as a polynomial in the Boolean vectors $\bm v_1, \bm v_2, \cdots \bm v_K$: 
	\begin{align*}
	&\E_{D_{\bm v_{\pi(1)} \dots \otimes \bm v_{\pi(k)}}} H_{\bm c}(\bm T) - \E_{D_0} q(\bm T) = \\& \hspace{5cm} \sum_{\bm c \in \tensor{\W^d}{k}: \|\bm c\|_1 \geq 1} \frac{\hat q(\bm c)}{\sqrt{d^{k \|\bm c\|_1} \cdot  \bm c!}} \cdot  \bm v_1^{\Tpar{\bm c}{1}[\pi]}  \cdots \cdot \bm v_K^{\Tpar{\bm c}{K}[\pi]}.
	\end{align*} 
	 Observing that $\Tpar{\bm c}{l} \in \{0,1\}^d$ for any $\bm c, l$ we can regroup the terms in the above summation as follows:
	\begin{align}
	\E_{D_{\bm v_{\pi(1)} \cdots \otimes \bm v_{\pi(k)}}} q(\bm T) - \E_{D_0} q(\bm T)  &= \sum_{\bm r_1,  \cdots \bm r_K \in \{0,1\}^d} \tilde{q}_\pi(\bm r_1, \bm r_2 \cdots \bm r_K) \cdot \bm v_1^{\bm r_1} \cdot \bm v_2^{\bm r_2} \cdots \cdot \bm v_K^{\bm r_K}, \label{rep_formula_1} \\ 
	\tilde{q}_\pi(\bm r_1, \cdots \bm r_K) & \explain{def}{=} \sum_{\substack{\bm c \in \tensor{\W^d}{k}: \|\bm c\|_1 \geq 1 \\ \Tpar{\bm c}{i}[\pi] = \bm r_i \; \forall i \in [K]}} \frac{\hat{q}(\bm c)}{\sqrt{d^{k \|\bm c\|_1}\cdot \bm c!}}. \label{rep_formula_2}
	\end{align}
	The above representation will allow us to leverage tools from Fourier analysis over the Boolean hypercube in the following step.
	
	\paragraph{Step 3, Hypercontractivity:} We define the following polynomial in $\bm v_1, \bm v_2 \cdots \bm v_K \in \{\pm 1\}^d$, for any $u \in \N$, $u \geq 2$.
	\begin{align*}
	    P_u(\bm v_1, \bm v_2 \cdots \bm v_K) & \explain{def}{=} \sum_{\bm r_1,  \cdots \bm r_K \in \{0,1\}^d}  (\sqrt{u-1})^{\|\bm r_1\|_1 + \cdots + \|\bm r_K\|_1} \cdot \tilde{q}_\pi(\bm r_{1:K}) \cdot \bm v_1^{\bm r_1} \cdot \bm v_2^{\bm r_2} \cdots \cdot \bm v_K^{\bm r_K},
	\end{align*}
	where the coefficients $\tilde{q}$ have been defined in Eq.~\ref{rep_formula_2}. Note that the bias $\E_{D_{\bm v_{\pi(1)} \dots \otimes \bm v_{\pi(k)}}} q(\bm T) - \E_{D_0} q(\bm T)$ is a smoothed out version of the polynomial $P$ (cf.~Eq.~\ref{rep_formula_1}) in the following sense: 
	\begin{align*}
	    	\E_{D_{\bm v_{\pi(1)} \dots \otimes \bm v_{\pi(k)}}}- \E_{D_0} q(\bm T)  &= \smooth{P}{\sqrt{u-1}}(\bm v_1, \bm v_2 \dots , \bm v_K),
	\end{align*}
    where, for any $\lambda \geq 1$, the operator $\smooth{\cdot}{\lambda}$ down weights the coefficient of $\bm v_1^{\bm r_1} \cdots \bm v^{\bm r_K}_K$ by $\lambda^{- (\|\bm r_1\|_1 +   \cdots \|\bm r_K\|_1)}$ (see Appendix \ref{fourier_boolean_appendix} for more background). $(2,u)$-Hypercontractive inequalities allow us to upper bound the $L_u$ norm of $\smooth{P}{\sqrt{u-1}}(\bm v_1, \bm v_2 \cdots , \bm v_K)$ in terms of the $L_2$ norm of $P(\bm v_1, \bm v_2 \cdots \bm v_K)$. Applying the $(2,u)$-Hypercontractive inequality (Fact~\ref{hypercontractive_fact}), we obtain,
    \begin{align}
        \E |\E_{D_{\bm v_{\pi(1)} \cdots \otimes \bm v_{\pi(k)}}} - \E_{D_0} q(\bm T)|^u & \leq   \left(\E P_u^2(\bm v_1, \bm v_2 \cdots \bm v_K) \right)^{\frac{u}{2}}.
    \label{hypercontractive_conclusion}
    \end{align}
    Note that the monomial functions $\bm v_1^{\bm r_1} \bm v_2^{\bm r_2} \cdots \bm v_K^{\bm r_K}$ are orthonormal when $\bm v_i \explain{i.i.d.}{\sim} \unif{\{\pm 1\}^d}$ (See again Appendix~\ref{fourier_boolean_appendix}). Hence, by Parseval's identity and Eq.~\ref{rep_formula_2},
    \begin{align*}
        &\E P_u^2(\bm v_1, \bm v_2 \cdots \bm v_K)  = \sum_{\bm r_1,  \cdots \bm r_K \in \{0,1\}^d}  (u-1)^{\|\bm r_1\|_1 + \|\bm r_2\| \cdots + \|\bm r_K\|_1} \cdot \tilde{q}^2_\pi(\bm r_1, \bm r_2 \cdots \bm r_K) \\
        & \explain{}{=}  \sum_{\bm r_1,  \cdots \bm r_K \in \{0,1\}^d}  (u-1)^{\|\bm r_1\|_1 + \|\bm r_2\| \cdots + \|\bm r_K\|_1} \cdot \left( \sum_{\substack{\bm c \in \tensor{\W^d}{k}: \|\bm c\|_1 \geq 1 \\ \Tpar{\bm c}{i}[\pi] = \bm r_i \; \forall i \in [K]}} \frac{\hat{q}(\bm c)}{\sqrt{d^{k \|\bm c\|_1}\cdot \bm c!}} \right)^2.
    \end{align*}
    Applying the Cauchy-Schwarz inequality,
    \begin{align*}
        &\E P_u^2(\bm v_{1:K})  \leq \\ &\hspace{0.2cm} \sum_{\bm r_1,  \cdots \bm r_K}  (u-1)^{\|\bm r_1\|_1 +  \cdots + \|\bm r_K\|_1} \cdot \left( \sum_{\substack{\bm c:  \|\bm c\|_1 \geq 1 \\ \Tpar{\bm c}{i}[\pi] = \bm r_i \; \forall i}} \frac{1}{d^{k \|\bm c\|_1}\cdot \bm c!} \right) \cdot \left(  \sum_{\substack{\bm c: \|\bm c\|_1 \geq 1 \\ \Tpar{\bm c}{i}[\pi] = \bm r_i \; \forall i}} \hat{q}^2(\bm c)\right).
    \end{align*}
    Note that, since $\E_{D_0} q^2(\bm T) = 1$,
    \begin{align*}
        \sum_{\bm r_1,  \cdots \bm r_K} \sum_{\substack{\bm c: \|\bm c\|_1 \geq 1 \\ \Tpar{\bm c}{i}[\pi] = \bm r_i \; \forall i}} \hat{q}^2(\bm c) \leq 1.
    \end{align*}
    Hence by the $\ell_1$/$\ell_\infty$ H\"older's inequality:
    \begin{align*}
        \E P_u^2(\bm v_1, \bm v_2 \cdots \bm v_K) & \leq \max_{\bm r_1, \bm r_2 \cdots \bm r_K \in \{0,1\}^d}  (u-1)^{\|\bm r_1\|_1 +  \cdots + \|\bm r_K\|_1}  \cdot \left( \sum_{\substack{\bm c:  \|\bm c\|_1 \geq 1 \\ \Tpar{\bm c}{i}[\pi] = \bm r_i \; \forall i}} \frac{1}{d^{k \|\bm c\|_1}\cdot \bm c!} \right).
    \end{align*}
    By a relabelling of $[d]$, the expression inside the maximum in the above display depends on $\bm r_1, \bm r_2 \cdots \bm r_K$ only through $l_1 \explain{def}{=} \|\bm r_1\|_1, l_2 \explain{def}{=} \|\bm r_2\|_1 \cdots , l_K \explain{def}{=} \|\bm r_K\|_1$. Hence we have,
    \begin{align*}
         \E P_u^2(\bm v_1, \bm v_2 \cdots \bm v_K) & \leq \max_{l_1,l_2 \cdots l_K \in \W} \left( e \cdot  (u-1)^{l_1 + l_2 \dots l_K} \cdot p_\pi(l_1,l_2 \cdots l_K) \right),
    \end{align*}
    where,
    \begin{align*}
        p_\pi(l_1,l_2 \cdots l_K) & \explain{def}{ =} \frac{1}{e} \sum_{\substack{\bm c:  \|\bm c\|_1 \geq 1 \\ \Tpar{\bm c}{i}[\pi] =(\bm{1}_{l_i}, \bm 0)\; \forall \;  i \;  \in \;  [K]}} \frac{1}{d^{k \|\bm c\|_1}\cdot \bm c!}.
    \end{align*}
    Substituting the above bound on the second moment in Eq.~\ref{hypercontractive_conclusion}, we obtain,
    \begin{align*}
        \E |\E_{D_{\bm v_{\pi(1)} \dots \otimes \bm v_{\pi(k)}}} q(\bm T) - \E_{D_0} q(\bm T)|^u & \leq \left( \max_{l_1,l_2 \cdots l_K \in \W} \left( e \cdot  (u-1)^{l_1 + l_2 \dots l_K} \cdot p_\pi(l_1,l_2 \dots l_K) \right) \right)^{\frac{u}{2}}.
    \end{align*}
    Finally Markov's inequality (cf.~Eq.~\ref{markov_inequality}) gives us,
    \begin{align*}
        &\P\left[ \left| \E_{D_{\bm v_{\pi(1)} \dots \otimes \bm v_{\pi(k)}}} q(\bm T) - \E_{D_0} q(\bm T) \right| > \epsilon \right]  \leq \\ & \hspace{5cm} \left( \frac{1}{\epsilon^2} \cdot \max_{l_1,l_2 \dots l_K \in \W} \left( e \cdot  (u-1)^{l_1 + l_2 \dots l_K} \cdot p_\pi(l_1,l_2 \dots l_K) \right) \right)^{\frac{u}{2}}.
    \end{align*}
    This concludes the proof of the proposition for the reference measure $D_0$. The result for the reference measure $\overline{D}(\pi)$ follows via a similar calculation and we only sketch the argument for this case. Note that if $\bm T \sim D_0$, then,
    \begin{align*}
        \bm T + \frac{\overline{\bm V}(\pi)}{\sqrt{d^k}} \sim \overline{D}(\pi).
    \end{align*}
    Now consider any query function such that $\E_{\overline{D}(\pi)} q^2(T)  = 1$. This means that:
    \begin{align*}
        \E_{D_0} q^2 \left( \bm T + \frac{\overline{\bm V}(\pi)}{\sqrt{d^k}} \right) = 1.
    \end{align*}
    Hence we can expand the function $q(\frac{\overline{\bm V}(\pi)}{\sqrt{d^k}} + \cdot) $ in the multivariate Hermite Basis:
    \begin{align*}
        q \left(\frac{\overline{\bm V}(\pi)}{\sqrt{d^k}} + \bm T \right) & = \sum_{\bm c} \hat{q}(\bm c) H_{\bm c}(\bm T).
    \end{align*}
    Hence,
    \begin{align*}
        \E_{\overline{D}(\pi)} q(\bm T)  = \hat{q}(\bm 0), \;  \E_{\overline{D}(\pi)} q^2(\bm T) = \sum_{\bm c} \hat{q}(\bm c)^2 = 1.
    \end{align*}
    Define the tensor $\bm V(\pi)$ as:
    \begin{align*}
        \bm V(\pi) & = \bm v_{\pi(1)} \otimes \bm v_{\pi(2)} \dots \otimes \bm v_{\pi(k)}.
    \end{align*}
    Also note that when $\bm T \sim D_0$, then,
    \begin{align*}
        \bm T + \frac{\bm V(\pi)}{\sqrt{d^k}} \sim D_{\bm v_{\pi(1)} \otimes \bm v_{\pi(2)} \dots \otimes \bm v_{\pi(k)}}.
    \end{align*}
    Hence, 
    \begin{align*}
        \E_{D_{\bm v_{\pi(1)} \otimes \bm v_{\pi(2)} \dots \otimes \bm v_{\pi(k)}}} q(\bm T) & = \E_{D_0}  q\left( \bm T + \frac{\bm V(\pi)}{\sqrt{d^k}} \right) \\
        & = \E_{D_0} \left[ \sum_{\bm c} \hat{q}(\bm c) H_{\bm c} \left( \bm T + \frac{\bm V(\pi) - \overline{\bm V}(\pi)}{\sqrt{d^k}} \right) \right].
    \end{align*}
    Due to the fact that each entry of $\bm V(\pi)$ is either 1 or marginally Rademacher distributed and $\E \bm V(\pi) = \overline{\bm V}(\pi)$, we have,
    \begin{align*}
        V(\pi)_{i_1,i_2 \dots ,i_k} - \overline{V}(\pi)_{i_1,i_2 \dots , i_k} & = \begin{cases} 0 &:   (i_1,i_2 \dots i_k) \in \supp{\overline{\bm V}(\pi)} \\  V(\pi)_{i_1,i_2 \dots ,i_k} &: (i_1,i_2 \dots i_k) \not\in \supp{\overline{\bm V}(\pi)} \end{cases}.
    \end{align*}
    Using Fact \ref{hermite_special_property}, we can compute,
    \begin{align*}
        \E_{D_0}  H_{\bm c} \left( \bm T + \frac{\bm V(\pi) - \overline{\bm V}(\pi)}{\sqrt{d^k}} \right) & = \begin{cases} 0 &: \supp{\bm c} \cap \supp{\overline{\bm V}(\pi)} \neq \phi \\ \frac{\bm V(\pi)^{\bm c}}{\sqrt{\bm c!d^{k \|\bm c\|_1}}} &: \supp{\bm c} \cap \supp{\overline{\bm V}(\pi)} = \phi \end{cases}.
    \end{align*}
    Hence we obtain the following Fourier expansion of the bias $\E_{D_{\bm v_{\pi(1)} \otimes \bm v_{\pi(2)} \dots \otimes \bm v_{\pi(k)}}} q(\bm T) - \E_{D_0} q(\bm T)$:
    \begin{align*}
	&\E_{D_{\bm v_{\pi(1)} \dots \otimes \bm v_{\pi(k)}}} H_{\bm c}(\bm T) - \E_{\overline{D}(\pi)} q(\bm T)= \\ &\hspace{4cm} \sum_{\substack{\bm c \in \tensor{\W^d}{k}: \|\bm c\|_1 \geq 1\\ \supp{\bm c} \cap \supp{\overline{\bm V}(\pi)} = \phi}} \frac{\hat q(\bm c)\cdot  \bm v_1^{\Tpar{\bm c}{1}[\pi]} \cdots \cdot \bm v_K^{\Tpar{\bm c}{K}[\pi]}}{\sqrt{d^{k \|\bm c\|_1} \cdot  \bm c!}} .
	\end{align*} 
	From here on wards the arguments made previously for the case when the reference distribution was $D_0$ can be repeated verbatim and give the second claim of the proposition. 
    \end{proof}
\section{Asymptotic Analysis of $p_\pi$ and $\overline{p}_\pi$}
Recall that Proposition \ref{fourier_analytic_prop} showed that the number of queries required to solve the Tensor PCA problems was intimately connected to the asymptotics of the combinatorial coefficients $p_\pi(l_1,l_2 \dots , l_K)$ and $\overline{p}_\pi(l_1,l_2, \dots ,l_K)$. We had also arrived at a probabilistic interpretation of these coefficients as:
\begin{align*}
p_\pi(l_1,l_2 \dots , l_K) & = \P \left( \|\bm C\|_1 \geq 1, \Tpar{\bm C}{i}[\pi] = (\bm{1}_{l_i}, \bm 0) \; \forall \; i \; \in [K]  \right),
\end{align*}
and,
\begin{multline*}
     \overline{p}_\pi(l_1,l_2,\dotsc, l_K) \\ = \P \left( \|\bm C\|_1 \geq 1, \Tpar{\bm C}{i}[\pi] = (\bm{1}_{l_i}, \bm 0) \; \forall \; i \; \in [K], \; \supp{\bm C} \cap \supp{\overline{\bm V}(\pi)} = \phi  \right),
\end{multline*}
where $\bm C \in \tensor{\W^d}{k}$ was a tensor with i.i.d. $\pois{d^{-k}}$ entries. This goal of this section is to prove Proposition \ref{asymptotic_proposition} which provides an upper bound on these combinatorial coefficients. This section is organized as follows:
\begin{itemize}
    \item Section \ref{rademacher_formula_proof} contains the proof of Lemma \ref{rademacher_formula_proof} which provides a convenient analytic formula for the probabilities $\P ( \Tpar{\bm C}{i}[\pi] = (\bm{1}_{l_i}, \bm 0) \; \forall \; i \; \in [K]  )$.
    \item Section \ref{asymptotics} uses this formula to prove Proposition \ref{asymptotic_proposition}.
\end{itemize}
\subsection{Proof of Lemma \ref{representation_formula_rademacher}}
\label{rademacher_formula_proof}
\begin{proof}
The proof of this lemma relies on the following simple observation. Let $Y$ be an arbitrary random variable taking values in $\W$. Let $X$ be a Rademacher random variable independent of $Y$. Then,
\begin{align*}
    \Indicator{\text{Y has even parity}} = \E [X^Y | Y], \; \Indicator{\text{Y has odd parity}}  = \E [X^{Y+1} | Y] .
\end{align*}
Appealing to this observation, we have,
\begin{align*}
    \Indicator{ \Tpar{\bm C}{i}[\pi] = (\bm{1}_{l_i}, \bm 0) \; \forall \; i \; \in [
    K]}& = \E \left[ \bm X_1^{\Tsum{\bm C}{1}[\pi]} \cdot  \cdots \bm X_K^{\Tsum{\bm C}{K}[\pi]}  \prod_{i=1}^K \prod_{j=1}^{l_i} X_{ij}\; \bigg| \; \bm C \right].
\end{align*}
Recalling Eq.~\ref{tsum_notation}, the above expression can be written as:
\begin{align*}
    \Indicator{ \Tpar{\bm C}{i}[\pi] = (\bm{1}_{l_i}, \bm 0) \; \forall \; i \; \in [
    K]} & = \E \left[ \bm X_{\pi(1)}^{\Tsum{\bm C}{1}} \cdot   \cdots \bm X_{\pi(k)}^{\Tsum{\bm C}{k}}  \prod_{i=1}^K \prod_{j=1}^{l_i} X_{ij}\; \bigg| \; \bm C \right].
\end{align*}
Hence,
\begin{align*}
    &\P\left( \Tpar{\bm C}{i} = (\bm{1}_{l_i}, \bm 0) \; \forall \; i \; \in [K] \right) \\& \hspace{1cm}= \E \left[ \bm X_{\pi(1)}^{\Tsum{\bm C}{1}} \cdot  \bm X_{\pi(2)}^{\Tsum{\bm C}{2}} \cdot  \cdots \bm X_{\pi(k)}^{\Tsum{\bm C}{k}} \cdot \prod_{i=1}^K \prod_{j=1}^{l_i} X_{ij}\right] \\
    &\hspace{1cm} =  \E \E \left[ \bm X_{\pi(1)}^{\Tsum{\bm C}{1}} \cdot  \bm X_{\pi(2)}^{\Tsum{\bm C}{2}} \cdot  \cdots \bm X_{\pi(k)}^{\Tsum{\bm C}{k}} \cdot \prod_{i=1}^K \prod_{j=1}^{l_i} X_{ij} \bigg| \; \|\bm C \|_1, \bm X_1, \bm X_2 \cdots \bm X_K\right] \\
    & \hspace{1cm} \explain{(a)}{=} \E \left[ \left(\prod_{i=1}^K \E \left[ \bm X_{\pi(i)}^{\Tsum{\bm C}{i}} \; \bigg| \|\bm C\|_1, \bm X_{\pi(i)}\;   \right] \right) \cdot \prod_{i=1}^K \prod_{j=1}^{l_i} X_{ij} \right] \\
    & \hspace{1cm} \explain{(b)}{=} \E \left[ (\overline{X}_1^{s_1} \cdot \overline{X}_2^{s_2} \cdot \cdots \overline{X}_K^{s_K})^{\|\bm C\|_1} \cdot \prod_{i=1}^K \prod_{j=1}^{l_i} X_{ij} \right] \\
    & \hspace{1cm} \explain{(c)}{=} \frac{1}{e} \cdot \E \left[ e^{\overline{X}_1^{s_1} \cdot \overline{X}_2^{s_2} \cdot \cdots \overline{X}_K^{s_K}} \prod_{i=1}^K \prod_{j=1}^{l_i} X_{ij}  \right]. 
\end{align*}
In the above display, the Equation marked (a) uses the conditional independence implied by Proposition \ref{poisson_tensor_proposition}. In the step marked (b), we used the formula for the generating function of a Multinomial distribution (Fact \ref{multinomial_distribution_properties}). The equality (c) uses the formula for the generating function of a Poisson distribution (Fact \ref{poisson_distribution_properties}). This proves the claim of the lemma.
\end{proof}

\subsection{Proof of Proposition \ref{asymptotic_proposition}} \label{asymptotics}
This section is devoted to the proof of Proposition \ref{asymptotic_proposition}. Recall that Lemma \ref{representation_formula_rademacher} relates $p_{\pi}, \overline{p}_\pi$ to expectations of the form:
\begin{align*}
    \E \left[ e^{\overline{X}_1^{s_1} \cdot \overline{X}_2^{s_2} \cdot \dots \overline{X}_K^{s_K}} \prod_{i=1}^K \prod_{j=1}^{l_i} X_{ij}  \right].
\end{align*}
Our strategy to analyze the above expectation will to approximate the exponential by its Taylor expansion of a suitable degree. Recalling the independence of the Rademacher vectors $\bm X_{1:K}$, this reduces our task to analyze expectations of the form:
\begin{align*}
    \E  \left[ \overline{X}^s \prod_{i=1}^l X_i \right].
\end{align*}
This is the content of the following intermediate lemma. 
\begin{lemma} \label{basic_bounds}
Let $\bm X \in \{\pm 1\}^d$ be a random vector with i.i.d.~Rademacher entries. Let $\overline{X}$ be the average of the entries of $\bm X$. Then, if $l>s$ or if $l+s$ has odd parity, then,
\begin{align*}
  \E  \left[ \overline{X}^s \prod_{i=1}^l X_i \right] & = 0.
\end{align*}
Otherwise,
\begin{align*}
    0 \leq \E  \left[ \overline{X}^s \prod_{i=1}^l X_i \right] & \leq  \frac{C_s}{\sqrt{d^{l+s}}},
\end{align*}
where $C_s$ denotes a constant depending only on $s$.
\end{lemma}
\begin{proof}
When $l+s$ has odd parity, then $\overline{X}^s \prod_{i=1}^l X_i$ is an odd polynomial in $\bm X$ and hence, $$\E[\overline{X}^s \prod_{i=1}^l X_i] = 0, \; \text{if $l+s$ is odd}.$$
When $l>s$, note that $\overline{X}^s$ is a degree $s$ polynomial, and since $s<l$, it is orthogonal to $\prod_{i=1}^l X_i$, hence,
$$\E[\overline{X}^s \prod_{i=1}^l X_i] = 0, \; \text{if $l>s$}.$$
Finally when $l \leq s$, observe that $\E[\overline{X}^s \prod_{i=1}^l X_i]$ is simply the coefficient of $\prod_{i=1}^l X_i$ in the expansion of $\overline{X}^s$ in the monomial basis for Boolean functions (see Appendix \ref{fourier_boolean_appendix}). Expanding $\overline{X}^s$ using the Multinomial theorem shows that it is a polynomial with all non-negative coefficients. This gives us the claim:
\begin{align*}
    \E[\overline{X}^s \prod_{i=1}^l X_i] & \geq 0.
\end{align*}
In order to understand its asymptotics, we decompose $\overline{X}$ as:
\begin{align*}
    \overline{X} & = \frac{Z}{d} + \frac{d-l}{d} \overline{Y}, \; Z \explain{def}{=} X_1 + X_2 + \dots + X_l, \; \overline{Y} = \frac{ \sum_{i=l+1}^d X_i}{d-l}.
\end{align*}
Hence,
\begin{align*}
    \E \overline{X}^s \prod_{i=1}^l X_i & = \sum_{t=0}^s \binom{s}{t} \cdot \left(\frac{d-l}{d}\right)^{s-t} \cdot \frac{1}{d^t} \cdot \E \left[ Z^t  \prod_{i=1}^l X_i\right] \cdot \E \overline{Y}^{s-l}.
\end{align*}
We observe that (using the same argument made previously), when $t < l$, $Z^t$ is a degree $t$ polynomial in $X_1,X_2 \dots X_l$ and hence orthogonal to $\prod_{i=1}^l X_i$:
\begin{align*}
    \E \left[ Z^t  \prod_{i=1}^l X_i\right] & = 0.
\end{align*}
Hence,
\begin{align}
\label{formula_basic_lemma}
     \E \overline{X}^s \prod_{i=1}^l X_i & = \sum_{t=l}^s \binom{s}{t} \cdot \left(\frac{d-l}{d}\right)^{s-t} \cdot \frac{1}{d^t} \cdot \E \left[ Z^t  \prod_{i=1}^l X_i\right] \cdot \E \overline{Y}^{s-t}
\end{align}
Next we observe that,
\begin{align*}
     |Z^t  \prod_{i=1}^l X_i| & \leq l^t,
\end{align*}
Since $\sqrt{d-l} \cdot  \overline{Y}$ is 1-subgaussian (since it is a normalized sum of i.i.d.~Rademacher variables), we have,
\begin{align*}
    \E |\overline{Y}^{s-t}| & \leq \frac{1}{\sqrt{(d-l)^{s-t}}} \cdot 2^{s-t} \cdot (s-t)^{\frac{s-t}{2}} \leq \frac{C_s}{\sqrt{d^{s-t}}},
\end{align*}
where $C_s$ denotes a constant depending only on $s$.
Substituting these estimates in the Eq.~\ref{formula_basic_lemma}, we obtain,
\begin{align*}
     \E \overline{X}^s \prod_{i=1}^l X_i & \leq \frac{C_s}{\sqrt{d^{s+t}}},
\end{align*}
for some constant $C_s$ depending only on $s$. This concludes the proof of the lemma. 
\end{proof}
We now present the proof of Proposition \ref{asymptotic_proposition} restated below for convenience. 
\asympprop*
\begin{proof} 
First recall the definitions of the coefficients $p_\pi$ and $\overline{p}_\pi$: 
\begin{align*}
     p_\pi(l_{1:K}) & = \P \left( \|\bm C\|_1 \geq 1, \Tpar{\bm C}{i}[\pi] = (\bm{1}_{l_i}, \bm 0) \; \forall \; i \; \in [K] \right), \\
     \overline{p}_\pi(l_{1:K}) & = \P \left( \|\bm C\|_1 \geq 1, \Tpar{\bm C}{i}[\pi] = (\bm{1}_{l_i}, \bm 0) \; \forall \; i \; \in [K], \; \supp{\bm C} \cap \supp{\overline{\bm V}(\pi)} = \phi  \right).&
\end{align*}
The above expressions immediately give us $\overline{p}_\pi(l_1,l_2 \dots l_K) \leq {p}_\pi(l_1,l_2 \dots l_K)$. Next we split our analysis into 4 cases: 
\begin{description}
\item [Case 1: $l_1 \vee l_2 \vee \dots l_K = 0$.] In this case, Lemma \ref{representation_formula_rademacher} gives us:
\begin{align*}
    \P\left( \Tpar{\bm C}{i}[\pi] =\bm 0 \; \forall \; i \; \in [K] \right) & = \frac{1}{e} \E \left[ e^{\overline{X}_1^{s_1} \cdot \overline{X}_2^{s_2} \cdot \dots \overline{X}_K^{s_K}} \right],
\end{align*}
where $s_i$ were defined as:
\begin{align*}
    s_i & \explain{def}{=} | \pi^{-1}(i) | = |\{j \in [k]: \pi(j) = i\}|.
\end{align*}
Using the inequality $|e^x - 1 - x|  \leq  \frac{e x^2}{2} \; \forall \; x \leq 1$ one obtains:
\begin{align*}
    &\P\left( \Tpar{\bm C}{i}[\pi] =\bm 0 \; \forall \; i \; \in [K] \right) \\&\hspace{3.8cm}\leq \frac{1}{e} \left(1 + \E \overline{X}_1^{s_1} \cdot \overline{X}_2^{s_2} \cdot \dots \overline{X}_K^{s_K} + \frac{e \cdot \E \overline{X}_1^{2 s_1} \cdot \overline{X}_2^{2 s_2} \cdot \dots \overline{X}_K^{2 S_K}}{2} \right) \\
    &\hspace{3.8cm} = \frac{1}{e} + \frac{1}{2 \cdot d^{k}}.
\end{align*}
By Lemma \ref{basic_bounds},
\begin{align*}
    \E \overline{X}_1^{2 s_1} \cdot \overline{X}_2^{2 s_2} \cdot \dots \overline{X}_K^{2 S_K} & \leq \frac{C_k}{d^{s_1 + s_2 + \dots S_K}} \leq \frac{C_k}{d^k}.
\end{align*}
Recall that $\oddness$ denoted the number of $i \in [K]$ such that $s_i$ is odd. If $\oddness \geq 1$, then,
\begin{align*}
     \E \overline{X}_1^{s_1} \cdot \overline{X}_2^{s_2} \cdot \dots \overline{X}_K^{s_K} & = 0.
\end{align*}
Otherwise, if $\oddness = 0$, then Lemma \ref{basic_bounds} gives us,
\begin{align*}
    \left| \E \overline{X}_1^{s_1} \cdot \overline{X}_2^{s_2} \cdot \dots \overline{X}_K^{s_K} \right| & \leq \frac{C_k}{\sqrt{d^k}}.
\end{align*}
Hence,
\begin{align*}
     \P\left( \Tpar{\bm C}{i}[\pi] =\bm 0 \; \forall \; i \; \in [K] \right) & \leq  \begin{cases} \frac{1}{e} + \frac{C_k}{\sqrt{d^k}} &: \oddness = 0 \\ 
     \frac{1}{e} + \frac{C_k}{d^k} &: \oddness \geq 1\end{cases}.
\end{align*}
We use the above estimates to bound $p_\pi(0,0, \dots 0)$ and $\overline{p}_\pi(0,0 \dots , 0)$.  Recall that,
\begin{align*}
    p_\pi(0,0, \dots 0) & = \P \left( \|\bm C\|_1 \geq 1, \Tpar{\bm C}{i}[\pi] = \bm 0 \; \forall \; i \; \in [K]  \right) \\
    & = \P \left( \Tpar{\bm C}{i}[\pi] = \bm 0 \; \forall \; i \; \in [K]  \right) - \P \left( \|\bm C\|_1 =0  \right) \\ &\explain{(a)}{=} \P \left( \Tpar{\bm C}{i}[\pi] = \bm 0 \; \forall \; i \; \in [K]\right)  - \frac{1}{e} \\
    & \leq \begin{cases} \frac{C_k}{\sqrt{d^k}} &: \oddness = 0 \\ 
     \frac{C_k}{d^k} &: \oddness \geq 1\end{cases}.
\end{align*}
In the equation marked (a) we used the fact that $\|\bm C\|_1 \sim \pois{1}$ (cf.~Proposition \ref{poisson_tensor_proposition}).
Next we analyze $\overline{p}_\pi(0,0, \dots 0)$. We have,
\begin{align*}
     &\overline{p}_\pi(0,\dots 0) \\& \hspace{0.3cm}= \P \left( \|\bm C\|_1 \geq 1, \Tpar{\bm C}{i}[\pi] = \bm 0 \; \forall \; i \; \in [K], \;   \supp{\bm C} \cap \supp{\overline{\bm V}(\pi)} = \phi\right) \\
    &\hspace{0.3cm} = \P \left( \Tpar{\bm C}{i}[\pi] = \bm 0 \; \forall \; i \; \in [K], \; \supp{\bm C} \cap \supp{\overline{\bm V}(\pi)} = \phi  \right) - \P \left( \|\bm C\|_1 =0  \right) \\ &\hspace{0.3cm}\explain{}{=} \P \left( \Tpar{\bm C}{i}[\pi] = \bm 0 \; \forall \; i \; \in [K], \; \supp{\bm C} \cap \supp{\overline{\bm V}(\pi)} = \phi  \right)  - \frac{1}{e}.
\end{align*}
Note that, 
\begin{align*}
    \supp{\bm C} \cap \supp{\overline{\bm V}(\pi)} = \phi \Leftrightarrow \sum_{(i_1,i_2 \dots i_k) \in \supp{\overline{\bm V}(\pi)} } C_{i_1,i_2 \dots , i_k} = 0.
\end{align*}
Furthermore, the random variables,
\begin{align*}
    \sum_{(i_1,i_2 \dots i_k) \in \supp{\overline{\bm V}(\pi)} } C_{i_1,i_2 \dots , i_k} \text{ and }  \;\Tpar{\bm C}{\cdot}[\pi], 
\end{align*}
are independent since any $C_{i_1,i_2 \dots i_k}$ with $(i_1, i_2 \dots , i_k) \in \supp{\overline{\bm V}(\pi)}$ is counted an even number of times in $\Tsum{\bm C}{\cdot}[\pi]$ and hence does not influence the parities $\Tpar{\bm C}{\cdot}[\pi]$. Hence,
\begin{align*}
     &\overline{p}_\pi(0,0, \dots 0)  =\\ &\hspace{1cm} \P \left( \Tpar{\bm C}{i}[\pi] = \bm 0 \; \forall \; i \; \in [K] \right) \cdot \P \left(  \sum_{(i_1,\dots i_k) \in \supp{\overline{\bm V}(\pi)} } C_{i_1,i_2 \dots , i_k} = 0 \right)  - \frac{1}{e}.
\end{align*}
Furthermore, 
\begin{align*}
    \sum_{(i_1,i_2 \dots i_k) \in \supp{\overline{\bm V}(\pi)} } C_{i_1,i_2 \dots , i_k} & \sim \pois{\frac{|\supp{\overline{\bm V}(\pi)}|}{d^k}},
\end{align*}
and, since the entries of $\overline{\bm V}(\pi)$ are in $\{0,1\}$:
\begin{align*}
    \frac{1}{d^k} \cdot |\supp{\overline{\bm V}(\pi)}| & = \frac{1}{d^k} \sum_{i_1,i_2 \dots  , i_k} \overline{\bm V}(\pi)_{i_1,i_2 \dots i_k} \\
    & = \frac{1}{d^k} \sum_{i_1,i_2 \dots  , i_k} (\E_{\bm X_{1:K} \sim \unif{\{\pm 1\}^d}} \bm X_{\pi(1)} \otimes \bm X_{\pi(2)} \dots \otimes \bm X_{\pi(k)})_{i_1,i_2 \dots i_k} \\
    & = \E \overline{X}_1^{s_1} \cdot \overline{X}_2^{s_2} \dots \overline{X}_K^{s_K}.
\end{align*}
Hence,
\begin{align*}
    \P \left(  \sum_{(i_1,\dots i_k) \in \supp{\overline{\bm V}(\pi)} } C_{i_1,i_2 \dots , i_k} = 0 \right) & = \exp \left( - \E \overline{X}_1^{s_1} \cdot \overline{X}_2^{s_2} \dots \overline{X}_K^{s_K}  \right).
\end{align*}
This gives us,
\begin{align*}
    &\overline{p}_\pi(0,0, \dots 0)\\ &= \P \left( \Tpar{\bm C}{i}[\pi] = \bm 0 \; \forall \; i \; \in [K] \right) \cdot \P \left(  \sum_{(i_1,\dots i_k) \in \supp{\overline{\bm V}(\pi)} } C_{i_1,i_2 \dots , i_k} = 0 \right)  - \frac{1}{e} \\
    & = \frac{1}{e} \left( \E \exp\left( \overline{X}_1^{s_1} \cdot \overline{X}_2^{s_2} \dots \overline{X}_K^{s_K} - \E \overline{X}_1^{s_1} \cdot \overline{X}_2^{s_2} \dots \overline{X}_K^{s_K} \right) - 1\right) \\
    & \leq \frac{1}{e} \cdot \left(   \E(\overline{X}_1^{s_1}  \dots \overline{X}_K^{s_K} - \E \overline{X}_1^{s_1} \dots \overline{X}_K^{s_K}) + \frac{e\E( \overline{X}_1^{s_1}  \dots \overline{X}_K^{s_K} - \E \overline{X}_1^{s_1} \dots \overline{X}_K^{s_K})^2}{2}   \right) \\
    & \leq \frac{C_k}{d^k}
\end{align*}
In the last step, we again appealed to Lemma \ref{basic_bounds}. Hence,
\begin{align*}
     \overline{p}_\pi(0,0, \dots 0) & \leq \frac{C_k}{d^k}.
\end{align*}
This concludes the analysis of this case. 
\item [Case 2: $l_i \leq s_i  \; i = 1,2 \dots , K, \; l_1 \vee l_2 \dots  \vee l_K \geq 1$. ] Recall that,
\begin{align*}
     p_\pi(l_1,l_2 \dots , l_K) & = \P \left( \|\bm C\|_1 \geq 1, \Tpar{\bm C}{i}[\pi] = (\bm{1}_{l_i}, \bm 0) \; \forall \; i \; \in [K] \right).
\end{align*}
Since $l_1 \vee l_2 \dots \vee l_K \geq 1$, can remove the condition $\|\bm C\|_1 \geq 1$:
\begin{align*}
    p_\pi(l_1,l_2 \dots , l_K) & = \P \left( \Tpar{\bm C}{i}[\pi] = (\bm{1}_{l_i}, \bm 0) \; \forall \; i \; \in [K] \right).
\end{align*}
By Lemma \ref{representation_formula_rademacher}, we have,
\begin{align*}
    p_\pi(l_1,l_2 \dots , l_K) & =  \frac{1}{e} \E \left[ e^{\overline{X}_1^{s_1} \cdot \overline{X}_2^{s_2} \dots \overline{X}_K^{s_K}} \prod_{i=1}^K \prod_{j=1}^{l_i} X_{ij} \right].
\end{align*}
Using the bound:
\begin{align*}
    1 + x  \leq e^x \leq 1 + x + \frac{e x^2}{2}, \; \forall \; x \leq 1,
\end{align*}
we obtain,
\begin{align*}
    p_\pi(l_1, \dots l_K) & \leq \frac{1}{e} \cdot \left( \prod_{i=1}^K  \prod_{j=1}^{l_i} \E X_{ij}  + \prod_{i=1}^k \E \left[ \overline{X}_i^{s_i} \prod_{j=1}^{l_i} X_{ij} \right] + \frac{e}{2} \cdot \prod_{i=1}^K \E \left[ \overline{X}_i^{2 s_i} \right] \right).
\end{align*}
We note that since $l_1 \vee l_2 \dots \vee l_K \geq 1$,
\begin{align*}
    \prod_{i=1}^K  \prod_{j=1}^{l_i} \E X_{ij} & = 0.
\end{align*}
If $s_i + l_i$ is even for each $i \in [K]$, then, Lemma \ref{basic_bounds} gives us,
\begin{align*}
    \prod_{i=1}^k \E \left[ \overline{X}_i^{s_i} \prod_{j=1}^{l_i} X_{ij} \right] \leq \frac{C_k}{\sqrt{d^{k + \sum_{i=1}^K l_i}}}, \; \prod_{i=1}^K \E \left[ \overline{X}_i^{2 s_i} \right] \leq \frac{C_k}{d^k}, \;  p_\pi(l_1, \dots l_K) & \leq \frac{C_k}{\sqrt{d^{k + \sum_{i=1}^K l_i}}}.
\end{align*}
On the other hand, if $s_i + l_i$ is odd even for a single $i$, Lemma \ref{basic_bounds} gives us,
\begin{align*}
    \prod_{i=1}^k \E \left[ \overline{X}_i^{s_i} \prod_{j=1}^{l_i} X_{ij} \right] =0, \; \prod_{i=1}^K \E \left[ \overline{X}_i^{2 s_i} \right] \leq \frac{C_k}{d^k}, \;  p_\pi(l_1, \dots l_K) & \leq \frac{C_k}{d^k}.
\end{align*}
Hence, when $l_i \leq s_i  \; i = 1,2 \dots , K, \; l_1 \vee l_2 \dots  \vee l_K \geq 1$, we have shown,
\begin{align*}
    p_\pi(l_1, \dots l_K) & \leq \begin{cases} C_k d^{-\frac{k + \sum_{i=1}^K l_i}{2}} &: \text{$l_i + s_i$ is even $\forall \; i \; \in [K]$} \\ C_k d^{-k} &: \text{otherwise} \end{cases}.
\end{align*}
This concludes the analysis for this case. 
\item [Case 3: $ \exists \; i \; : l_i > s_i$.] By the case definition, we have,
\begin{align*}
    \max_{i \in [K]} \bigg\lceil \frac{l_i}{s_i} \bigg\rceil \geq 2.
\end{align*}
Without loss of generality, we can assume,
\begin{align*}
    \max_{i \in [K]} \bigg\lceil \frac{l_i}{s_i} \bigg\rceil & =  \bigg\lceil \frac{l_1}{s_1} \bigg\rceil.
\end{align*}
Let $t \geq 1$, $t \in \N$ be such that:
\begin{align*}
    t s_1 &< l_1 \leq (t+1) s_1.
\end{align*}
By Lemma \ref{representation_formula_rademacher}, we have,
\begin{align*}
    p_{\pi}(l_1,l_2 \dots l_K) & = \frac{1}{e} \cdot \E \left[ e^{\overline{X}_1^{s_1} \dots \overline{X}_K^{s_K}} \prod_{i=1}^K \prod_{j=1}^{l_i} X_{ij} \right].
\end{align*}
We use the bound:
\begin{align*}
    1 +  \sum_{a=1}^t \frac{x^a}{a!} - \frac{e |x|^{t+1}}{(t+1)!} & \leq e^x \leq 1 + \sum_{a=1}^t \frac{x^a}{a!} + \frac{e |x|^{t+1}}{(t+1)!}, \; |x| \leq 1.
\end{align*}
Hence,
\begin{align*}
    p_{\pi}(l_{1:K}) & \leq  \prod_{i=1}^K \prod_{j=1}^{l_i} \E X_{ij} + \sum_{a=1}^t \frac{1}{a!} \prod_{i=1}^K \E \left[ \overline{X}_i^{a s_i} \prod_{j=1}^{l_i} X_{ij} \right] + \frac{e}{(t+1)!} \prod_{i=1}^K \E[|\overline{X}_i|^{s_i(t+1)}].
 \end{align*}
 Since $l_1 > s_1 \geq 0 \implies l_1 \geq 1$, 
 \begin{align*}
     \prod_{i=1}^K \prod_{j=1}^{l_i} \E X_{ij} & = 0.
 \end{align*}
 For any $a \leq t$, $a s_1 \leq t s_1 < l_1$. Hence by Lemma \ref{basic_bounds},
 \begin{align*}
     \E \left[ \overline{X}_1^{a s_1} \prod_{j=1}^{l_1} X_{ij} \right] = 0 \implies \sum_{a=1}^t \frac{1}{a!} \prod_{i=1}^K \E \left[ \overline{X}_i^{a s_i} \prod_{j=1}^{l_i} X_{ij} \right] = 0.
 \end{align*}
 Finally again by Lemma \ref{basic_bounds},
 \begin{align*}
      p_{\pi}(l_1,l_2 \dots l_K) & \leq \frac{e}{(t+1)!} \prod_{i=1}^K \E[|\overline{X}_i|^{s_i(t+1)}] \leq \frac{C_k}{d^{(t+1)k/2}}.
 \end{align*}
 Note that,
 \begin{align*}
     (t+1) = \max_{i \in K} \bigg \lceil \frac{l_i}{s_i} \bigg \rceil \geq \max_{i \in K} \frac{l_i}{k} = \frac{l_1 \vee l_2 \dots \vee l_K}{k}.
 \end{align*}
 Hence,
 \begin{align*}
      p_{\pi}(l_1,l_2 \dots l_K) & \leq  \frac{C_k}{d^{\frac{l_1 \vee l_2 \dots \vee l_K}{2}}}.
 \end{align*}
 If $l_1 \vee l_2 \dots \vee l_K < 2k$, we can rely on the better bound (since $t \geq 1$):
 \begin{align*}
     p_{\pi}(l_1,l_2 \dots l_K) & \leq  \frac{C_k}{d^{k}}.
 \end{align*}
 Hence, if $ \exists \; i \; : l_i > s_i$, then,
 \begin{align*}
      p_{\pi}(l_1,l_2 \dots l_K) & \leq \begin{cases} C_k \cdot d^{-k} &: l_1 \vee l_2 \dots \vee l_K < 2k \\ C_k \cdot d^{-\frac{l_1 \vee l_2 \dots \vee l_K}{2}} &: l_1 \vee l_2 \dots \vee l_K \geq 2k \end{cases}.
 \end{align*}
\end{description}
Combining all of these cases gives us the claim of the proposition.
\end{proof}


\section{Proof of Theorem \ref{main_result}}
\label{main_result_proof}
In this section we show how the various auxiliary results we have proved so far imply Theorem \ref{main_result}.
\begin{proof}
In order to apply Proposition \ref{feldmans_sq_lb} we need to lower bound $\sdn{(3n)^{-\frac{1}{2}}}$. Lemma \ref{lb_on_sdn} gives us: 
\begin{align}
\label{sdn_lb_eq}
    &\sdn[D_0]{\frac{1}{\sqrt{3n}}}[\mathcal{H}(\pi)]  \nonumber\\&\hspace{2.5cm} \geq \frac{1}{4\sqrt{3 n}} \cdot \left( \sup_{q: \E_{D_0} q^2(\bm T) = 1} \P_{D \sim \unif{\mathcal H}(\pi)} \left[ \left| \E_{D} q(\bm T) - \E_{D_0} q(\bm T) \right| > \frac{1}{6\sqrt{n}}\right] \right)^{-1}.
\end{align}
Let $u \geq 2$ be a integer parameter that we will set suitably later. Proposition \ref{fourier_analytic_prop} gives us, for any $q$ such that $\E_{D_0} q^2(\bm T) = 1$,
\begin{align*}
     &\P_{D \sim \unif{\mathcal{H}(\pi)}} \left[ \left| \E_{D} q(\bm T) - \E_{D_0} q(\bm T) \right| > \frac{1}{6\sqrt{n}}\right]  \\& \hspace{4cm}\leq \left( 36 e n \cdot  \max_{l_1,l_2 \dots l_K \in \W} \left(  (u-1)^{l_1 + l_2 \dots l_K} \cdot p_\pi(l_{1:K})  \right) \right)^{\frac{u}{2}}.
\end{align*}
In order to upper bound $\max_{l_1,l_2 \dots l_K \in \W} \left(  (u-1)^{l_1 + l_2 \dots l_K} \cdot p_\pi(l_1,l_2 \dots l_K)  \right)$, we will utilize Proposition \ref{asymptotic_proposition}:
\begin{description}
\item [When $l_1 \vee l_2 \dots \vee l_K \geq  2k$: ] Proposition \ref{asymptotics} gives us the bound:
\begin{align*}
     (u-1)^{l_1 + l_2 \dots l_K} \cdot p_\pi(l_1,l_2 \dots l_K) & \leq  \frac{C_k \cdot (u-1)^{k l_1 \vee l_2 \dots \vee l_K}}{d^{l_1 \vee \dots l_K/2}} = \left( \frac{(u-1)^k}{\sqrt{d}} \right)^{l_1 \vee \dots l_K }.
\end{align*}
If $d \geq (u-1)^{2k}$, then the upper bound is a decreasing function of $l_1 \vee \dots l_K$ and is maximized at $l_1 \vee \dots l_K = 2k$.  Hence,
\begin{align*}
   \max_{l_1, \dots l_K: l_1 \vee \dots l_K \geq 2k} \;  (u-1)^{l_1 + l_2 \dots l_K} \cdot p_\pi(l_1,l_2 \dots l_K) & \leq \frac{C_k (u-1)^{2k^2}}{d^k}, \; \forall \; d \geq (u-1)^{2k}. 
\end{align*}
\item [When $l_1 \vee l_2 \dots \vee l_K < 2k$: ] Then,
\begin{align*}
     \max_{l_{1:K}: l_1 \vee \dots l_K \leq 2k} \;  (u-1)^{l_1 + l_2 \dots l_K} \cdot p_\pi(l_{1:K}) & \leq (u-1)^{2k^2} \cdot  \max_{l_{1:K}: l_1 \vee \dots l_K \leq 2k} \; p_\pi(l_{1:K}).
\end{align*}
By Proposition \ref{asymptotic_proposition}, when $l_1 \vee \dots l_K \leq 2k$,  the upper bound on $p_\pi(l_1,l_2 \dots l_K)$ is maximized by setting:
\begin{align*}
    l_i & = \begin{cases} 0 &: s_i \text{ is even} \\ 1&: s_i \text{ is odd} \end{cases}.
\end{align*}
Recall that the oddness parameter $\oddness$ is precisely the number $s_i$ that are odd, hence,
\begin{align*}
    \max_{l_1, \dots l_K: l_1 \vee \dots l_K \leq 2k} \;  (u-1)^{l_1 + l_2 \dots l_K} \cdot p_\pi(l_1,l_2 \dots l_K) & \leq \frac{C_k (u-1)^{2k^2}}{\sqrt{{d^{k+\oddness}}}}.
\end{align*}
\end{description}
Hence, we have obtained, 
\begin{align*}
    \max_{l_1,l_2 \dots l_K \in \W} \left(  (u-1)^{l_1 + l_2 \dots l_K} \cdot p_\pi(l_1,l_2 \dots l_K)  \right) & \leq \frac{C_k \cdot (u-1)^{2k^2}}{\sqrt{d^{k+\oddness}}},
\end{align*}
provided $d \geq (u-1)^{2k}$.
Hence,
\begin{align*}
    \sup_{q: \E_{D_0} q^2(\bm T) = 1} \P_{D \sim \unif{\mathcal H(\pi)}} \left[ \left| \E_{D} q(\bm T) - \E_{D_0} q(\bm T) \right| > \frac{1}{6\sqrt{n}}\right] & \leq \left( \frac{C_k \cdot (u-1)^{2k^2} \cdot n}{\sqrt{d^{k+\oddness}}}\right)^{\frac{u}{2}}
\end{align*}
Substituting this in Eq.~\ref{sdn_lb_eq}, we obtain,
\begin{align}
\label{sdn_lb_eq_2}
    \sdn[D_0]{\frac{1}{\sqrt{3n}}}[\mathcal{H}(\pi)] & \geq \frac{1}{8 \sqrt{n}} \cdot \left( \frac{1}{ C_k \cdot  (u-1)^{2k^2}} \cdot \frac{\sqrt{d^{k+\oddness}}}{n} \right)^u,
\end{align}
provided $d \geq (u-1)^{2k}$. Now suppose there exist constants $C_0 \in (0,\infty)$ and $\epsilon \in (0,k)$ such that,
\begin{align*}
    n \leq C_0 d^{\frac{k+\oddness}{2}-\epsilon}.
\end{align*}
Consider any $L \in \N$ and suppose that,
\begin{align}
    d \geq 11 \cdot \left(\frac{4L + k+\oddness}{4 \epsilon} \right)^{2k} \explain{def}{=} C_1(k,L,\epsilon).
\end{align}
We take $u = \epsilon^{-1} \cdot (L + (k+\oddness)/4)$ in Eq.~\ref{sdn_lb_eq_2}. This gives us,
\begin{align*}
    \sdn[D_0]{\frac{1}{\sqrt{3n}}}[\mathcal{H}(\pi)] & \geq c_1(k,L,\epsilon,C_0) \cdot d^L,
\end{align*}
where, 

\begin{align}
    \change{c_1(k,L,\epsilon,C_0) \explain{def}{=} \frac{1}{8} \left( \frac{\epsilon}{(1+C_0^2) \cdot C_k \cdot L}\right)^{\frac{2Lk^3}{\epsilon}}.}
\end{align}

Hence, by Proposition \ref{feldmans_sq_lb}, any  SQ algorithm which solves the Tensor PCA testing problem with $n$ samples must make at least $ c_1(k,L,\epsilon C_0) \cdot d^L$ queries. In order to obtain an estimation lower bound, we appeal to Proposition \ref{feldmans_sq_lb} which lets us  conclude that any  SQ algorithm which solves the Tensor PCA estimation problem with $n$ samples must make at least $ 0.5 \cdot c_1(k,L,\epsilon C_0) \cdot d^L$ queries.

In the case when $\oddness = 0$, we can obtain a stronger estimation lower bound, by applying Proposition \ref{feldmans_sq_lb} with the base measure $\overline{D}(\pi)$. Lemma \ref{lb_on_sdn} and Proposition \ref{fourier_analytic_prop} tell us that $	\sdn[D_\star]{\epsilon}[\mathcal{H}(\pi)]$ can be controlled in terms of the coefficients $\overline{p}_\pi(l_1,l_2 \dots ,l_K)$. Furthermore from Proposition \ref{asymptotic_proposition} that in general,
\begin{align*}
    \overline{p}_\pi(l_1, \dots l_K) & \leq  p_{\pi}(l_1, \dots l_K),
\end{align*}
but when $l_1 = l_2 \dots = l_K = 0$, we have the improved upper bound $\overline{p}_\pi(0,0 \dots , 0) \leq C_k \cdot d^{-k}$ (compared to $p_\pi(0,0, \dots 0) \leq C_k d^{-k/2}$). Consequently the quantity:
\begin{align*}
     \max_{l_1,l_2 \dots l_K \in \W} \left(  (u-1)^{l_1 + l_2 \dots l_K} \cdot \overline{p}_\pi(l_1,l_2 \dots l_K)  \right),
\end{align*}
is maximized by setting one $l_i = 2$ and the remaining ones as $0$. This gives us,
\begin{align*}
    \max_{l_1,l_2 \dots l_K \in \W} \left(  (u-1)^{l_1 + l_2 \dots l_K} \cdot \overline{p}_\pi(l_1,l_2 \dots l_K)  \right) & \leq \frac{C_k \cdot (u-1)^{2k^2}}{\sqrt{d^{k+2}}},
\end{align*}
which results in an improved lower bound:
\begin{align*}
    \sdn[D_\star]{\epsilon}[\mathcal{H}(\pi)]  \geq \frac{1}{8 \sqrt{n}} \cdot \left( \frac{1}{ C_k \cdot  (u-1)^{2k^2}} \cdot \frac{\sqrt{d^{k+2}}}{n} \right)^u,
\end{align*}
from which we can conclude via Proposition \ref{feldmans_sq_lb} that if $n \ll \sqrt{d^{k+2}}$, estimation is impossible in the SQ model with polynomially many queries.
\end{proof}


\section{Optimal Statistical Query Procedures} \label{optimal_SQ_appendix}
In this section we describe SQ procedures that achieve the optimal sample complexity for the Tensor PCA testing and estimation problems. These procedures leverage the ``partial trace'' technique. This approach was first introduced by \citet{hopkins2016fast} who used this technique to compress matrices obtained by flattening higher order tensors in the context of symmetric Tensor PCA with $k=3$ and some other planted inference problems. A different variant of this technique was used by \citet{anandkumar2017homotopy} as a initialization procedure for tensor power method in the context of symmetric Tensor PCA with $k=3$. Recent work by \citet{biroli2019iron} generalized this approach for Symmetric Tensor PCA with arbitrary $k$. 

We show that procedures based on the partial trace technique can be extended to Tensor PCA with partial symmetries and can be implemented in the SQ model via Fact \ref{scalar_mean_estimation_fact}. When implemented in the SQ model, these procedures suffer from a degradation in performance (when compared to the non-SQ performance on symmetric Tensor PCA problems from prior work \citep{hopkins2016fast,anandkumar2017homotopy,biroli2019iron}). However, the lower bounds of the previous section show that this is unavoidable.

We first design an SQ Test for even-order Symmetric Tensor PCA. 
\begin{lemma} [SQ Test for Even Symmetric Tensor PCA]  \label{symmetric_even_SQ_test} If $n \gg d^{\frac{k}{2}}$, there is a SQ algorithm which solves the Symmetric Tensor PCA testing problem with even $k$ with $O(\log(n))$ queries. 
\end{lemma}
\begin{proof}
Let $k = 2l$. Consider the query function:
\begin{align*}
    q(\bm T) & = \sum_{i_1,i_2 \dots i_l} T_{i_1,i_1,i_2,i_2 \dots i_l, i_l}.
\end{align*}
It is easy to check that for any $\bm v$ with $\|\bm v\| = \sqrt{d}$, under $D_{\bm v^{\otimes k}}$, $q(\bm T) \sim \gauss{1}{d^{\frac{k}{2}}}$ where as, under $D_0$, $q(\bm T) \sim \gauss{0}{d^{\frac{k}{2}}}$. Hence for any $D \in \{D_0\} \cup \{D_{\bm v^{\otimes k}}: \bm v \in \R^d, \; \|\bm v\| = \sqrt{d} \}$,
\begin{align*}
    \operatorname{Var}_D q(\bm T)  = 1, \; \E_{D} q^2(\bm T) \leq d^{\frac{k}{2}} + 1.
\end{align*}
Consequently, we can use Fact \ref{scalar_mean_estimation_fact} to construct an SQ procedure which returns an estimate of $\E_D q(T)$ denoted by $\hatE q(\bm T)$ such that,
\begin{align*}
    |\hatE_D q(\bm T) - \E_D q(\bm T) | \rightarrow 0, \text{ as}  \; d \rightarrow \infty,
\end{align*}
if $n \gg d^{\frac{k}{2}}$. This requires only $O(\log(d))$ queries. This can be turned into a test for Tensor PCA by declaring $D = D_0$ if $\hatE_D q(\bm T) \leq 0.5$ and $D \in \{D_{\bm v^{\otimes k}}: \bm v \in \R^d, \; \|\bm v\| = \sqrt{d} \}$ if $\hatE_D q(\bm T) > 0.5$. The correctness of this test follows by the error guarantee of $\hatE_D q(\bm T)$.
\end{proof}

Next we consider testing and estimation in generalized Tensor PCA with partial symmetries. We introduce the notion of the standard form of a labelling function. 

\begin{definition} Recall that for a labelling function $\pi: [K] \rightarrow [k]$, we defined the oddness parameter of $\pi$ (denoted by $\oddness$) as the number of labels in $[K]$ which are used an odd number of times. Given a labelling function $\pi$, we can construct $l \explain{def}{=} (k-\oddness)/2$ pairs from $k-\oddness$ elements of $[k]$ by pairing two indices $i,j \in [k]$ if $\pi(i) = \pi(j)$. The remaining $\oddness$ indices are left unpaired. A labelling function $\pi$ is in standard form if $\{2l + 1, 2l +2 \dots , k\}$ are the unpaired indices and,
\begin{align*}
    \pi(2i - 1) & = \pi(2i), \; i \in \left[l\right]. 
\end{align*}
If $\pi$ is in standard form, then for any $D \in \mathcal{D}(\pi)$, 
\begin{align*}
    \E_D \bm T & = \bm E_1 \otimes \bm E_2 \dots \otimes \bm E_{l} \otimes \bm O
\end{align*}
where $\bm E_1, \dots ,\bm E_l$ are rank-1 symmetric matrices with $\|\bm E_i\| = \Tr(\bm E_i) =  1$ and $\bm O$ is a $\oddness$-order rank 1 tensor with $\|\bm O\| = 1$.
\end{definition}

\begin{lemma} There exists a SQ algorithm which given the knowledge of the labelling function $\pi$ solves the Tensor PCA estimation problem after making at most $O(d^k \cdot \log(n) )$ queries provided the sample size $n$ satisfies:
\begin{align*}
    n  \gg \begin{cases} \sqrt{d^{k+2}} &: \oddness = 0 \\ \sqrt{d^{k+\oddness}} &: \oddness \geq 1 \end{cases}.
    \end{align*}
Furthermore when $\oddness \geq 1$, there is a test based on this estimator that solves Tensor PCA testing problem under the same assumptions on the sample size $n$.
\end{lemma}

\begin{proof}
Since the algorithm knows the labelling function $\pi$, we can assume, without loss of generality that $\pi$ is in standard form since by applying a suitable permutation to the modes of the tensor, the algorithm can always convert the Tensor PCA problem to the standard form. We define $l \explain{def}{=} (k-\oddness)/2$, where $\oddness$ is the number of labels used an odd number of times by $\pi$. We first consider the Tensor PCA estimation problem. Recall that by the definition of standard form, for any $D \in \mathcal{D}(\pi)$,
\begin{align*}
     \E_D \bm T & = \bm E_1 \otimes \bm E_2 \dots \otimes \bm E_{l} \otimes \bm O, \; \forall \; D \; \in \mathcal{D}(\pi).
\end{align*}
\begin{description}
\item [Estimating $\bm O$: ] We construct an estimator for $\bm O$ by estimating it entry-wise. Consider the query functions:
\begin{align*}
    q_{i_1,i_2 \dots i_\oddness}(\bm T) & = \sum_{j_1,j_2 \dots , j_l}  \bm T_{j_1,j_1,j_2,j_2 \dots j_l,j_l, i_1,i_2 \dots i_{\oddness}}.
\end{align*}
For any $D \in \mathcal{D}(\pi)$
\begin{align*}
    q_{i_1,i_2 \dots i_\oddness}(\bm T) \sim \gauss{\Tr(\bm E_1) \cdot \Tr(\bm E_2) \cdot  \dots \Tr(\bm E_l) O_{i_1,i_2 \dots ,i_{\oddness}}}{d^{l}} = \gauss{O_{i_1,i_2 \dots ,i_{\oddness}}}{d^{l}}.
\end{align*}
Hence using Fact \ref{scalar_mean_estimation_fact}, one can obtain estimates such that,
\begin{align*}
    \left|\hatE_D  q_{i_1,i_2 \dots i_\oddness}(\bm T) - \E_D  q_{i_1,i_2 \dots i_\oddness}(\bm T) \right| & \leq  9 \cdot \log(n) \cdot \sqrt{\frac{d^l}{n}}, \; \forall \; D \; \in \;  \mathcal{D}(\pi).
\end{align*}
Now we construct an estimate of $\bm O$, denoted by $\hat{\bm O}$ with $\hat{O}_{{i_1,i_2 \dots i_\oddness}} = \hatE_D  q_{i_1,i_2 \dots i_\oddness}(\bm T) $. This estimate has the error guarantee:
\begin{align*}
    \|\hat{\bm O} - \bm O\|^2 & = \sum_{i_1,i_2 \dots i_\oddness} (\hat{O}_{i_1,i_2 \dots i_\oddness} - O_{i_1,i_2 \dots ,i_\oddness})^2 \\
    & \leq 81 \cdot \log^2(n) \cdot \frac{d^{l+\oddness}}{n} = 81 \cdot \log^2(n) \cdot \frac{\sqrt{d^{k+\oddness}}}{n}. 
\end{align*}
Under the sample size assumption, $\|\hat{\bm O} - \bm O\| \rightarrow 0$ as $d \rightarrow \infty$. This further implies $\|\hat{\bm O}\| \rightarrow \|\bm O\| = 1$ and, $\ip{\hat{\bm O}}{\bm O} \rightarrow 1$. Note that when $\oddness = 0$, we can simply skip this step. 
\item [Estimating $\bm E_{1:l}$] We estimate $\bm E_1$ entry-wise by using the queries:
\begin{align*}
    q_{ab}(\bm T) & = \frac{1}{\|\hat{\bm O}\|} \cdot \sum_{j_1,j_2,\dots j_{l-2}, i_1,i_2 \dots i_\oddness} T_{a,b,j_1,j_1,j_2,j_2 \dots ,j_l,j_l, i_1,i_2 \dots i_\oddness} \cdot \hat{O}_{i_1,i_2 \dots i_\oddness}, \; a,b \in [d].
\end{align*}
Under $D \in \mathcal{D}(\pi)$, 
\begin{align*}
    q_{ab}(\bm T) & \sim \gauss{(E_1)_{a,b} \cdot \frac{\ip{\bm O}{\hat{\bm O}}}{\|\hat{\bm O}\|}}{d^{l-1}}.
\end{align*}
Hence using Fact \ref{scalar_mean_estimation_fact}, we can obtain estimates $\hatE_{D} q_{ab}(\bm T)$ with the guarantee:
\begin{align*}
    \left|\hatE_D  q_{ab}(\bm T) - \E_D  q_{ab}(\bm T) \right| & \leq  9 \cdot \log(n) \cdot \sqrt{\frac{d^{l-1}}{n}}, \; \forall \; D \; \in \;  \mathcal{D}(\pi).
\end{align*}
We construct an estimator of $\bm E_1$, denoted by $\hat{\bm E}_1$ with entries defined by $(\hat{E}_1)_{a,b} = \hatE_D q_{ab}(\bm T)$.  The error guarantee of this estimator is given by:
\begin{align*}
    \left\|\hat{\bm E}_1 -  \frac{\ip{\bm O}{\hat{\bm O}}}{\|\hat{\bm O}\|} \bm E_1\right\|^2 & = 81 \cdot \log^2(n) \cdot \frac{d^{l+1}}{n} = 81 \cdot \log^2(n) \cdot \frac{\sqrt{d^{k-\oddness + 2}}}{n}.
\end{align*}
By Triangle Inequality,
\begin{align*}
    \left\|\hat{\bm E}_1 -   \bm E_1\right\| & \leq \left\|\hat{\bm E}_1 -  \frac{\ip{\bm O}{\hat{\bm O}}}{\|\hat{\bm O}\|} \bm E_1\right\| + \left|1 - \frac{\ip{\bm O}{\hat{\bm O}}}{\|\hat{\bm O}\|} \right| = 9 \cdot \log(n) \cdot \left( \frac{\sqrt{d^{k-\oddness + 2}}}{n}\right)^{\frac{1}{2}} + o_d(1).
\end{align*}
When $\oddness = 0$, the sample size assumption $ n \gg \sqrt{d^{k+2}}$ guarantees $\|\bm E_1 - \hat{\bm E}_1 \| \rightarrow 0$ as $d \rightarrow \infty$. When $\oddness \geq 1$ since $k-\oddness + 2 \leq k + \oddness$, the assumption $n \gg \sqrt{d^{k + \oddness}} \implies n \gg \sqrt{d^{k-\oddness + 2}}$ and hence, $\|\bm E_1 - \hat{\bm E}_1 \| \rightarrow 0$. By a similar procedure, we can also estimate $\bm E_2, \bm E_3 \cdots \bm E_l$. Hence we have obtained estimators $\hat{\bm E}_{1:l}$ such that,
\begin{align*}
    \|\hat{\bm E}_i - \bm E_i\| \rightarrow 0, \; \|\hat{\bm E_i}\| \rightarrow 1, \; \ip{\bm E_i}{\hat{\bm E}_i} \rightarrow 1 \text{ as } d \rightarrow \infty. 
\end{align*}
\item [Estimating $\E_D q(\bm T)$] We construct an estimator of $\E_D \bm T$ as:
\begin{align*}
    \hatE_D \bm T &  = \frac{\hat{\bm E_1}}{\|\hat{\bm E_1}\|} \otimes \frac{\hat{\bm E_2}}{\|\hat{\bm E_2}\|} \dots \otimes \frac{\hat{\bm E_l}}{\|\hat{\bm E_l}\|} \otimes \frac{\hat{\bm O}}{\|\hat{\bm O}\|}.
\end{align*}
The error of this estimator can be computed as:
\begin{align*}
    \|\hatE_D \bm T - \E_D \bm T\|^2 & = 1 + 1 - 2 \cdot \left( \prod_{i=1}^l \frac{\ip{\hat{\bm E_i}}{\bm E_i}}{\|\hat{\bm E_i}\|} \right) \cdot \frac{\ip{\hat{\bm O}}{\bm O}}{\|\hat{\bm O}\|} \rightarrow 0 \text{ as } d \rightarrow \infty.
\end{align*}
\end{description}
This proves the claim of the lemma for estimation. For testing we observe that we only need to consider the case $\oddness \geq 1$ since Lemma \ref{symmetric_even_SQ_test} already handles the case $\oddness = 0$. Hence for any $D \in \mathcal{D}(\pi)$, the estimator $\hat{\bm O}$ satisfies $\|\hat{\bm O} \| \rightarrow 1$. When $D = D_0$, the same analysis can be repeated to show $\|\hat{\bm O}\| \rightarrow 0$. Hence we construct a test by declaring $D = D_0$ if $\|\hat{\bm O}\| \leq 0.5$ and $D \in \mathcal{D}(\pi)$ if $\|\hat{\bm O}\| > 0.5$. This test solves the Tensor PCA testing problem for large enough $d$.
\end{proof}


\section{Poisson Random Tensors}
This section is devoted to the proof of Proposition \ref{poisson_tensor_proposition} which describes a certain conditional independence structure in Poisson random tensors. We first collect some standard facts regarding the Multinomial and the Poisson distributions which will be useful for our analysis.
\label{Poisson_tensors_appendix}
\begin{definition}[Multinomial Distribution] \label{multinomial_distribution_definition} For any $n \in \N, k \in \N$, the Multinomial Distribution with parameters $(n,k)$ (denoted by $\mult{n}{k}$) is the distribution supported on the set:
\begin{align*}
\supp{\mult{n}{k}} & = \left\{ \bm x \in \R^k: \forall \; i \in [k], \; x_i \in \W, \; \sum_{i=1}^k x_i = n \right\},
\end{align*}
with the probability mass function given by:
\begin{align*}
\P \left( X_1 = x_1, X_2 = x_2 \dots , X_k = x_k \right) & = \binom{n}{x_1, x_2 \dots , x_k} \cdot  \frac{1}{k^{n}}, \; \forall \; \bm x  \in \supp{\mult{n}{k}}.
\end{align*}
\end{definition}
\begin{fact}[Properties of the Multinomial Distribution] \label{multinomial_distribution_properties} Let $\bm X \sim \mult{n}{k}$ and $\bm Y \sim \mult{m}{k}$ be two independent Multinomial random vectors. Then we have,
\begin{enumerate}
	\item The generating function of $\bm X$ is given by: 
	\begin{align*}
	\E \left[ z_1^{X_1} z_2^{X_2} \dots z_k^{X_k} \right] & = \left( \frac{\sum_{i=1}^k z_i}{k} \right)^n, \; \forall \bm z \in \C^k.
	\end{align*}
	\item The sum $\bm X + \bm Y$ is also a Multinomial random vector with $\bm X + \bm Y \sim \mult{n+m}{k}$. 
\end{enumerate}
\end{fact}
\begin{definition}[Poisson Distribution] \label{poisson_distribution_definition} The Poisson distribution with parameter $\lambda \geq 0$ denoted by $\pois{\lambda}$ is the distribution on $\W$ with the probability mass function:
\begin{align*}
\P \left( X = x \right) & = \frac{e^{-\lambda} \lambda^{x}}{x!}.
\end{align*}
\end{definition}
\begin{fact}[Properties of the Poisson Distribution] \label{poisson_distribution_properties} Let $X, X_1, X_2 \dots X_k$ be i.i.d. $\pois{\lambda}$ distributed random variables. Then,
\begin{enumerate}
	\item Generating Function of $\pois{\lambda}$: We have,
	\begin{align*}
	\E z^X & = e^{\lambda(z-1)}, \; \forall \; z \in \C.
	\end{align*}
	\item The sum $T = \sum_{i=1}^k X_i$ is also a Poisson random variable with $T \sim \pois{k\lambda}$. 
	\item The conditional distribution of the random vector $(X_1, X_2 \dots ,X_k)$ given $T$ is multinomial:
	\begin{align*}
	(X_1, X_2 \dots ,X_k) \; \big| \; T & \sim \mult{T}{k}.
	\end{align*}
\end{enumerate}
\end{fact}
We now prove the conditional independence property of Poisson tensors (Proposition \ref{poisson_tensor_proposition}) which we have restated below for convenience.  Let $\bm C$ be a random tensor in $\tensor{\bm W^d}{k}$ whose entries are sampled independently as follows:
\begin{align*}
    C_{i_1,i_2, \dots i_k} & \explain{i.i.d.}{\sim} \pois{\frac{1}{d^k}}.
\end{align*}
\poistensorprop*
\begin{proof} The claim that $ \|\bm C\|_1 \sim \pois{1}$ is a direct consequence of the additivity property of the Poisson Distribution (cf.~Fact \ref{poisson_distribution_properties}.3). We will prove the claim about the conditional distribution of $\Tsum{\bm C}{i}$ given $T$ by induction on the order of the tensor, $k$. 

\paragraph{Base Case, $k=2$:} Note that in this case $\bm C$ is a $d \times d$ matrix.  $\Tsum{\bm C}{1}$ is the vector of row-wise sums and $\Tsum{\bm C}{2}$ are the column wise sums. We first characterise the distribution of $\Tsum{\bm C}{1} \; | \; \|\bm C\|_1$. Due to the additivity property of the Poisson distribution (cf.~Fact \ref{poisson_distribution_properties}.2),
\begin{align*}
 \Tsum{\bm C}{1}_i & \explain{i.i.d.}{\sim} \pois{\frac{1}{d}}, \; \forall \; i \in [d]. 
\end{align*}
Since $\|\bm C\|_1 = \sum_{i=1}^d \Tsum{\bm C}{1}_i$, appealing to the Poisson-Multinomial connection (cf.~Fact \ref{poisson_distribution_properties}.3), we have,
\begin{align*}
\Tsum{\bm C}{1} \; | \; \|\bm C\|_1 & \sim \mult{\|\bm C\|_1}{d}.
\end{align*}
Next we characterize the conditional distribution of $\Tsum{\bm C}{2} \; | \; \|\bm C\|_1,  \Tsum{\bm C}{1}$. We first observe that,
\begin{align*}
(C_{i,1}, C_{i,2} \dots C_{i,d}) \; | \; \|\bm C\|_1 , \Tsum{\bm C}{1} & \explain{d}{=} (C_{i,1}, C_{i,2} \dots C_{i,n}) \; |\; \Tsum{\bm C}{1}_i \\
& \sim \mult{\Tsum{\bm C}{1}_i}{d}.
\end{align*}
In the last step, we again appealed to the Poisson-Multinomial connection (Fact \ref{poisson_distribution_properties}.3). Hence by the additivity property of the Multinomial distribution (Fact \ref{multinomial_distribution_properties}.2), we obtain,
\begin{align*}
\Tsum{\bm C}{2} \; | \; \|\bm C\|_1 , \Tsum{\bm C}{1} & = \sum_{i=1}^n (C_{i,1}, C_{i,2} \dots C_{i,n}) \; \bigg| \; \|\bm C\|_1 , \Tsum{\bm C}{1} \\
& \sim \mult{ \sum_{i=1}^n \Tsum{\bm C}{1}_i}{d} \\
& = \mult{\|\bm C\|_1}{d}.
\end{align*}
Hence,
\begin{align*}
\Tsum{\bm C}{1} \; | \; \|\bm C\|_1  \sim \mult{\|\bm C\|_1}{d}, \; \Tsum{\bm C}{2} \; | \; \|\bm C\|_1 , \Tsum{\bm C}{1} \sim \mult{\|\bm C\|_1}{d},
\end{align*}
which proves the proposition for $k=2$. 

\paragraph{Induction Hypothesis:} We assume the claim of the proposition holds for Poisson tensors of order $k-1$.

\paragraph{Induction Step:} We now prove the proposition for Tensors of order $k$. Since we plan to appeal to the induction hypothesis, we define a $k-1$ order tensor $\tilde{\bm{ C}}$ by summing over mode $k$ of $\bm C$:
\begin{align*}
\tilde{C}_{j_1,j_2 \dots j_{k-1}} & = \sum_{j_k} C_{j_1,j_2 \dots j_k},\; \forall  \; j_1, j_2 \dots j_{k-1} \in [d].
\end{align*}
We note that,
\begin{align*}
   \| \bm C\|_1 =\|\tilde{\bm C}\|_1, \; \Tsum{\bm C}{i} = \Tsum{\tilde{\bm C}}{i}, \; \forall \; i \; \in \; [k-1]. 
\end{align*}Hence by the induction hypothesis, we have,
\begin{align*}
\Tsum{\bm C}{i} \; | \; |\bm C\|_1 & \explain{i.i.d.}{\sim}  \mult{\|\bm C\|_1}{d}, \; \forall \; i \in [k-1].
\end{align*}
We now just need to characterize the distribution of $$\Tsum{\bm C}{k} \; | \; \|\bm C\|_1, \Tsum{\bm C}{1}, \Tsum{\bm C}{2} \dots \Tsum{\bm C}{k-1}.$$ We first study its distribution conditioned on a larger $\sigma$-algebra: $\Tsum{\bm C}{k} | \bm \tilde{\bm C}$. Observe that,
\begin{align*}
&\left(C_{j_1, \dots j_{k-1},1}, C_{j_1,  \dots j_{k-1},2}, \dots C_{j_1,  \dots j_{k-1},d} \right) \; \bigg| \; \tilde{\bm C} \\& \hspace{5cm} \explain{d}{=} \left(C_{j_1, j_2, \dots j_{k-1},1}, C_{j_1, j_2, \dots j_{k-1},2}, \dots C_{j_1, j_2, \dots j_{k-1},d} \right) \; \bigg| \; \tilde{C}_{j_1, \dots j_{k-1}} \\
&\hspace{5cm} \sim \mult{\tilde{C}_{j_1,j_2 \dots j_{k-1}}}{d}.
\end{align*}
In the last step we used the Multinomial-Poisson connection (Fact \ref{poisson_distribution_properties}.3). Hence, by additivity of Multinomial distributions (Fact \ref{multinomial_distribution_properties}.2),
\begin{align*}
\Tsum{\bm C}{k} \; | \; \bm \tilde{\bm C} & = \sum_{j_1,j_2 \dots j_{k-1}} \left(C_{j_1, \dots j_{k-1},1}, C_{j_1,  \dots j_{k-1},2}, \dots C_{j_1,  \dots j_{k-1},d} \right) \; \bigg| \; \tilde{\bm C} \\
& \sim \mult{\sum_{j_1,j_2 \dots j_{k-1}}\tilde{C}_{j_1,j_2 \dots j_{k-1}}}{d} \\
& = \mult{\|\bm C\|_1}{d}.
\end{align*}
This shows that,
\begin{align*}
\Tsum{\bm C}{k} \; | \; \|\bm C\|_1, \Tsum{\bm C}{1}, \Tsum{\bm C}{2} \dots \Tsum{\bm C}{k-1} & \sim \mult{\|\bm C\|_1}{d}.
\end{align*}
This concludes the induction step and the proof of the proposition. 
\end{proof}


\section{Proof of Proposition \ref{feldmans_sq_lb} and Lemma \ref{lb_on_sdn}}
\label{lb_on_sdn_appendix}
This section is dedicated to missing proofs from Appendix \ref{framework_appendix} on the framework for proving SQ lower bounds. 
\subsection{Proof of Proposition \ref{feldmans_sq_lb}}
\begin{proof}
As mentioned previously, the claim regarding the testing problem is simply Theorem 7.1 of \citet{feldman2018complexity}. The proof of the claim for estimation is a modification of their proof. We define $B$ as:
\begin{align*}
    B &\explain{def}{=} \sdn[D_\star]{(3n)^{-\frac{1}{2}}}[\mathcal{H}(\pi)].
\end{align*}
We will prove the contrapositive of this statement: For any SQ algorithm $\mathcal{A}$ which makes $t < B/2$ queries, we will exhibit a $D \in \mathcal{H}(\pi)$ and a corresponding $\vstat{D}{n}$ oracle which when used to answer queries made by $\mathcal{A}$ results in $\mathcal{A}$ outputting an estimator $\hat{\bm V}$ whose error satisfies $\|\hat{\bm V} - \E_{D} q(\bm T)\| \geq 0.5$ and hence $\mathcal{A}$ fails to solve the Tensor PCA estimation in the sense of Definition \ref{estimator_def}.

We begin by specifying the adversarial $\vstat{D}{n}$ oracle: This oracle avoids committing to a specific $D \in \mathcal{H}(\pi)$ in the beginning and simply responds to any query $q$ by the estimate $\E_{D_\star} q (\bm T)$. After the learning algorithm has exhausted its query budget by asking a sequence of queries $q_1, q_2 \dots , q_t$ and declares an estimator $\hat{\bm V}$, the oracle chooses a $D \in \mathcal{H}(\pi)$ to maximize the error of the estimate subject to the constraint that the responses are valid $\vstat{D}{n}$ responses. 
Define the sets:
\begin{align*}
    A_i & \explain{def}{=} \left\{ (\bm v_1, \bm v_2 \cdots \bm v_K): \bm v_j \in \{\pm 1\}^d  ,  \left| p_i(\bm v_{1:K}) - \E_{\mathcal{D}_\star} q_i(\bm T) \right| > \frac{1}{n} \vee \sqrt{\frac{p_i(\bm v_{1:K})-p_i^2(\bm v_{1:K})}{n}} \right\},
\end{align*}
where,
\begin{align*}
    p_i(\bm v_1, \bm v_2 \cdots \bm v_K) & \explain{def}{=} \E_{D_{\bm v_{\pi(1)} \cdots \otimes \bm v_{\pi(k)}}} q_i(\bm T).
\end{align*}
Then by the definition of a $\vstat{D}{n}$ oracle (Definition \ref{VSTAT_def}), for any $D$ such that:
\begin{align*}
    D \in \mathcal{H}(\pi) \setminus \bigcup_{i=1}^t A_i,
\end{align*}
the above specified adversarial oracle is a valid $\vstat{D}{n}$ oracle. We make the following claims:

\paragraph{Claim 1:} For all $i \in [t], \; |A_i| \leq 2^{dK}/B$.

\paragraph{Claim 2:} There exist $D_1, D_2$ such that, $D_1,D_2  \in \mathcal{H}(\pi) \setminus \bigcup_{i=1}^t A_i$ \&  $\| \E_{D_1} \bm T - \E_{D_2} \bm T \| \geq 1$.

We first assume the above claims are true and prove Proposition \ref{feldmans_sq_lb}. By the triangle inequality,
\begin{align*}
    \| \E_{D_1} \bm T - \E_{D_2}\bm T \|  \leq \| \E_{D_1} \bm T - \hat{\bm V} \| + \| \E_{D_2} \bm T - \hat{\bm V} \| \leq 2 \max(\| \E_{D_1} \bm T - \hat{\bm V} \| , \| \E_{D_2} \bm T - \hat{\bm V} \|).
\end{align*}
Appealing to Claim 2, we obtain,
\begin{align*}
    \max(\| \E_{D_1} \bm T - \hat{\bm V} \| , \| \E_{D_2} \bm T - \hat{\bm V} \|) & \geq \frac{1}{2}.
\end{align*}
Hence, we have found a $D \in \mathcal{D}(\pi)$ and a valid $\vstat{D}{n}$ oracle such that,
\begin{align*}
    \|\hat{\bm V} - \E_{D} \bm T\| & \geq \frac{1}{2},
\end{align*}
which was the claim of Proposition 2. Finally we provide a proof of the claims made above. 

\paragraph{Proof of Claim 1:} Exactly the same claim has been proved in the proof of Theorem 7.1 in \citet{feldman2018complexity}, so we refer the reader to this paper for the proof. 

\paragraph{Proof of Claim 2:} As the SQ algorithm makes more and more queries, the set of valid choices of $D$ for the oracle becomes smaller. We can track the set of valid choices of $D$ and the relative distances between them by a sequence of undirected graphs. The initial graph $G_0 = (V_0,E_0)$ has the vertex set $$V_0 =\{(\bm v_1, \bm v_2 \cdots \bm ,v_K): \bm v_j \in \{\pm 1\}^d, \; \forall \; j \; \in \; [K] \},$$ and the edge set:
\begin{align*}
    E_0 & = \{[(\bm u_1, \dots , \bm u_K), (\bm v_1,  \dots , \bm v_K)]: \bm u_i, \bm v_i \in \{\pm 1\}^d, |\ip{\bm u_i}{\bm v_i}| \leq 2^{-\frac{1}{k}}\cdot d \; \forall \; i \in [K]\}.
\end{align*}
Note that for any two connected vertices $(\bm u_1, \dots , \bm u_K)$ and $(\bm v_1,  \dots , \bm v_K)$:
\begin{align*}
    \|\E_{D_{\bm u_{\pi(1)} \dots \otimes \bm u_{\pi(k)}}} \bm T - \E_{D_{\bm v_{\pi(1)} \dots \otimes \bm v_{\pi(k)}}} \bm T\|^2 & = \frac{1}{d^k} \| \bm u_{\pi(1)} \dots \otimes \bm u_{\pi(k)}  - \bm v_{\pi(1)} \dots \otimes \bm v_{\pi(k)}\|^2 \\
    & = 2 - 2 \prod_{i=1}^k \frac{\ip{\bm u_{\pi(i)}}{\bm v_{\pi(i)}}}{d} \\
    & \geq 2 - 2 \cdot \left(\frac{1}{2^{\frac{1}{k}}}\right)^k = 1.
\end{align*}
Due to the underlying symmetry between the vertices, it is easy to check that $G_0$ is a $r$-regular graph for some $r \geq 1$. Hence $|E_0| = |V_0| \cdot r/2 = r \cdot 2^{Kd-1}$. After every query $q_t$, we update the graph as follows: The new graph $G_i = (V_i, E_i)$ where $V_i = V_{i-1} \setminus A_i $ and $E_{i}$ is the set of remaining edges after all edges which involved a vertex in $A_i$ are removed from $E_{i-1}$. Since each vertex has at most $r$ edges, we have,
\begin{align*}
    |E_{i-1}| - |E_{i}| & \leq r |A_i|  \explain{Claim 1}{\leq } \frac{r 2^{dK}}{B}.
\end{align*}
Hence,
\begin{align*}
    |E_0| - |E_{t}| & \leq \frac{r \cdot t \cdot 2^{dK}}{B}
\end{align*}
Since $t < B/2$, $|E_t| > |E_0| - r \cdot 2^{dK-1} = 0$. Hence $|E_t| \geq 1$, that is, the graph after $t$ queries contains at least one remaining edge. The vertices corresponding to this edge yield distributions $D_1,D_2$ satisfying Claim 2. 
\end{proof}
\subsection{Proof of Lemma \ref{lb_on_sdn}}
\begin{proof}
In order to show the desired lower bound on $\sdn[D_\star]{\epsilon}[\mathcal{H}(\pi)]$, it is sufficient to verify the implication in Eq.~\ref{sdn_implication} from Definition~\ref{sdn_definition} with
\begin{align*}
m & = \frac{\epsilon}{4} \cdot \left( \sup_{q: \E_{D_0} q^2(\bm T) = 1} \P_{D \sim \unif{\mathcal H}} \left[ \left| \E_{D} q(\bm T) - \E_{D_0} q(\bm T) \right| > \frac{\epsilon}{2}\right] \right)^{-1}.
\end{align*}
To this goal consider any $\mathcal{S} \subset \mathcal{H}(\pi)$ with $|\mathcal{S}| \geq m^{-1} \cdot |\mathcal{H}(\pi)|$ and any $q$ with $\E_{D_\star} q^2(\bm T) = 1$. First we observe that,
\begin{align}
\frac{\left| D \in \mathcal{S}: | \E_{D} q(\bm T) - \E_{D_\star} q(\bm T) | > \frac{\epsilon}{2} \right| }{|\mathcal{S}|} & \leq \frac{\left| D \in \mathcal{H}(\pi): | \E_{D} q(\bm T) - \E_{D_\star} q(\bm T) | > \frac{\epsilon}{2} \right| }{|\mathcal{S}|}\nonumber \\
& = \frac{|\mathcal{H}(\pi)|}{|\mathcal{S}|} \cdot \P_{D \sim \unif{\mathcal H(\pi)}} \left[ \left| \E_{D} q(\bm T) - \E_{D_\star} q(\bm T) \right| > \frac{\epsilon}{2}\right] \nonumber \\
& \leq \frac{\epsilon}{ 4} \label{fraction_bound}.
\end{align} 
Consider any $D \in \mathcal{H}(\pi)$. Hence, there exist hypercube vectors $\bm v_1, \bm v_2, \cdots \bm v_K \in \{\pm 1\}^d$ such that $D = D_{\bm v_{\pi(1)} \otimes v_{\pi(2)} \cdots \otimes v_{\pi(k)}}$. We first prove a worst case upper bound on $| \E_{D} q(\bm T) - \E_{D_\star} q(\bm T)|$ when $D_\star = D_0$:
\begin{align*}
|\E_{D} q(\bm T) - \E_{D_0} q (\bm T)| & \explain{}{=} \left| \E_{D_0} \left( \frac{\diff \P_{D}}{\diff \P_{D_0}}(\bm T) -1 \right) q(\bm T) \right| \\
& \explain{(a)}{\leq} \sqrt{\E_{D_0} q^2(\bm T)} \cdot \sqrt{\E_{D_0} \left( \frac{\diff \P_{D}}{\diff \P_{D_0}}(\bm T) -1 \right)^2} \\
& =  \sqrt{\E_{D_0} \left( \frac{\diff \P_{D}}{\diff \P_{D_0}}(\bm T)  \right)^2 - 1 } \\
& = \left( \E_{D_0} \exp \left( \|\bm T\|^2 - \|\bm T - d^{-\frac{k}{2}} \bm v_{\pi(1)} \otimes v_{\pi(2)} \dots \otimes v_{\pi(k)} \|^2 \right) - 1 \right)^{\frac{1}{2}} \\
& = \left( \E_{D_0} \exp \left( 2\ip{\bm T}{d^{-\frac{k}{2}}\bm v_{\pi(1)} \otimes v_{\pi(2)} \dots \otimes v_{\pi(k)}}  - 1\right) - 1 \right)^{\frac{1}{2}} \\
& \explain{(b)}{=} \sqrt{e-1} \leq 2.
\end{align*}
In the above display, in the equation marked (a), we applied Cauchy-Schwarz inequality and in the step marked (b), we noted that under $D_0$, the entries of $\bm T$ are i.i.d.~$\gauss{0}{1}$ and used the formula for the Gaussian moment generating function. Hence we have obtained, for any $D \in \mathcal{H}(\pi)$
\begin{align}
\label{worst_case_discripency}
|\E_{D} q(\bm T) - \E_{D_0} q (\bm T)| & \leq 2.
\end{align}
An exactly analogous calculation shows that:
\begin{align*}
    |\E_{D} q(\bm T) - \E_{\overline{D}(\pi)} q (\bm T)| & \leq \sqrt{e-1} \leq 2.
\end{align*}
We can now verify the implication in Eq.~\ref{sdn_implication}:
\begin{align*}
\frac{\sum_{D\in \mathcal{S}} \left| \E_{D} q(\bm T) - \E_{D_\star} q(\bm T) \right|}{|\mathcal{S}|}  & \leq  \frac{\epsilon}{2}  + \frac{\sum_{D: | \E_{D} q(\bm T) - \E_{D_\star} q(\bm T)| > \frac{\epsilon}{2}}  \left| \E_{D} q(\bm T) - \E_{D_\star} q(\bm T) \right|}{|\mathcal{S}|}  \\
& \explain{(a)}{\leq} \frac{\epsilon}{2} + \frac{\epsilon}{4} \cdot 2\\
& \leq \epsilon.
\end{align*}
In the above display, the step marked (a) is a consequence of Eqs.~\ref{fraction_bound} and~\ref{worst_case_discripency}. This verifies the implication in Eq.~\ref{sdn_implication} and completes the proof of this lemma. 
\end{proof}

\section{Fourier Analytic Tools}
This section collects some basic facts from Fourier analysis on the Boolean hypercube $\{\pm 1\}^d$ and the Gaussian Hilbert space $\mathcal{L}_2(\gauss{\bm 0}{\bm I_d})$. Our reference for all these results is the book of \citet{o2014analysis}.
\subsection{Fourier Analysis on $\{\pm 1\}^d$}
\label{fourier_boolean_appendix}
Consider the class of arbitrary functions $f: \{\pm 1\}^d \mapsto \R$:
\begin{align*}
\BoolSpace{d} & = \{ f: \{\pm 1\}^d \mapsto \R\}.
\end{align*} The following is a orthonormal basis for $\BoolSpace{d}$: 
\begin{align*}
\basis{\BoolSpace{d}} & = \left\{\bm z^{\bm r} \explain{def}{=} \prod_{i=1}^d z_i^{r_i}: \bm r \in \{0,1\}^d\right\}.
\end{align*}
The above basis is orthonormal in the following sense:
\begin{align*}
\E_{\unif{\pm 1}^{d}} \bm Z^{\bm r} \bm Z^{\bm s} & = \begin{cases} 0: & \bm r \neq \bm s\\ 1: & \bm r = \bm s \end{cases}.
\end{align*}
Hence given any function  $f \in \BoolSpace{d}$ has a unique representation of the form:
\begin{align*}
f(\bm z) & = \sum_{\bm r \in \{0,1\}^d} \hat{f}(\bm r)  \bm z^{\bm r}.
\end{align*}
The coefficients $\hat{f}(\bm r)$ are called the Fourier coefficients of $f$. 
This representation satisfies the usual Parseval's identity:
\begin{align*}
\E_{\unif{\{\pm 1\}^d}} f^2(\bm Z) & = \sum_{\bm r \in \{0,1\}^d } \hat{f}^2(\bm r).
\end{align*}
For any $\lambda \geq 1$, we define the operator $\smooth{\cdot}{\lambda}: \BoolSpace{d} \mapsto \BoolSpace{d}$ as follows:
\begin{align*}
\smooth{f}{\lambda}(\bm z) & \explain{def}{=} \sum_{\bm r \in \{0,1\}^d } \hat{f}(\bm r) \cdot \left( \frac{1}{\lambda} \right)^{\|\bm r\|_1} \cdot   \bm z^{\bm r}.
\end{align*}
The operator $\smooth{\cdot}{\lambda}$ smooths a function by down weighting the high degree Fourier coefficients. 
Hypercontractivity Theorems control $L_q$ norm of the random variable $\smooth{f}{\lambda}(\bm Z)$  in terms of $L_p$ norm of $f(\bm Z)$ where $\bm Z \sim \unif{\{\pm 1\}^d}$ where $q \geq p$ for a suitable value of $\lambda$. Our results use the $(2,q)$-Hypercontractivity Theorem stated below. 
\begin{fact}[$(2,q)$-Hypercontractivity Theorem] \label{hypercontractive_fact} For any $f \in \BoolSpace{d}$, we have,
	\begin{align*}
	\E_{\unif{\{\pm 1 \}^d}} \left[ \smooth{f}{\sqrt{q-1}}(\bm Z)^q \right]  & \leq  \left( \E_{\unif{\{\pm 1 \}^d}}[ f^2(\bm Z)] \right)^{\frac{q}{2}}.
	\end{align*}
\end{fact}  
\subsection{Fourier Analysis on the Gauss Space }
\label{fourier_gauss_appendix}
Consider the functional space $\GaussSpace{d}$ defined as follows:
\begin{align*}
\GaussSpace{d} & \explain{def}{=} \left\{f: \R^d \mapsto \R, \E_{\gauss{\bm 0}{\bm I_d}} f^2(\bm Z) < \infty \right\}.
\end{align*}
The multivariate Hermite polynomials for a complete orthonormal basis for $\GaussSpace{d}$.  These are defined as follows: for any $\bm c \in \W^d$, define,
\begin{align*}
H_{\bm c} (\bm z) & = \prod_{i=1}^d H_{c_i} (z_i).
\end{align*}
In the above display, for any $k \in \W$, $H_k$ are the usual (univariate) orthonormal Hermite polynomials with the property $$\E_{\gauss{0}{1}} H_k(Z) H_l(Z) = \begin{cases} 0: & k \neq l \\ 1: & k=l \end{cases}. $$
The orthonormality property is inherited by the multivariate Hermite polynomials and we have,
$$\E_{\gauss{\bm 0}{\bm I_d}} H_{\bm c}(\bm Z) H_{\bm d}(\bm Z) = \begin{cases} 0: & \bm c \neq \bm d \\ 1: & \bm c=\bm d \end{cases}. $$
Since these polynomials form an orthonormal basis of $\GaussSpace{d}$ any $f \in \GaussSpace{d}$ admits an expansion of the form: 
\begin{align*}
f(\bm z) & = \sum_{\bm c \in (\mathbb N \cup \{0\})^d} \hat{f}(\bm c) H_{\bm c}(\bm z).
\end{align*}
In the above display $\hat{f}(\bm c) \in \R$ are the Fourier coefficients of $f$. They satisfy the usual Parseval's relation:
\begin{align*}
\sum_{\bm c \in \W^d} \hat f^2(\bm c) & = \E_{\gauss{\bm 0}{\bm I_d}} f^2(\bm Z).
\end{align*}
A particular desirable property of the univariate Hermite polynomials is the following: for any $\mu \in \R, k \in \W$ we have,
\begin{align*}
\E_{\gauss{0}{1}} H_k(\mu + Z) & = \frac{\mu^k}{\sqrt{k!}}.
\end{align*}
This implies the following property of multivariate Hermite polynomials which will be particularly useful for us:
\begin{fact} \label{hermite_special_property} For any $\bm \mu \in \R^d$ and any $\bm c \in \W^{d}$, we have,
\begin{align*}
\E_{\gauss{\bm 0}{\bm I_d}} H_k(\bm \mu + \bm Z) & = \frac{\bm \mu ^{\bm c}}{\sqrt{\bm c!}}.
\end{align*}
In the above display, we are using the following notation:
\begin{align*}
\bm \mu^{\bm c} \explain{def}{=} \prod_{i=1}^m \mu_i^{c_i}, \; \bm c! \explain{def}{=} \prod_{i=1}^m( c_i !).
\end{align*}
\end{fact}

\bibliographystyle{plainnat}
\bibliography{ref}

\end{document}